\documentclass[a4paper,11pt]{amsart}
 
\usepackage{fullpage}
\usepackage{tikz}
\usetikzlibrary{matrix,arrows,decorations.pathmorphing}
\usepackage{amsmath}
\numberwithin{equation}{section}
\usepackage{amssymb}
\usepackage{bbm}
\usepackage{comment}
\usepackage{amsthm}
\usepackage[all,cmtip]{xy}
\usepackage[toc,page]{appendix}
\usepackage[british]{babel}
\usepackage{courier}
\usepackage{enumerate}
\usepackage[T1]{fontenc}
\usepackage{listings}
\theoremstyle{plain}
\newtheorem{theorem}{Theorem}[section]
\newtheorem*{theorem*}{Theorem}

\newtheorem{lemma}[theorem]{Lemma}
\newtheorem{proposition}[theorem]{Proposition}
\newtheorem{cor}[theorem]{Corollary}
\theoremstyle{remark}
\newtheorem{remark}[theorem]{Remark}

\theoremstyle{definition}
\newtheorem*{ack}{Acknowledgements}

\title{Homotopy types of gauge groups related to $S^3$-bundles over $S^4$}
\author{Ingrid Membrillo-Solis}
\address{Mathematical Sciences, University of Southampton, University Road, Southampton SO17~1BJ, United Kingdom}

\date{\today}

\begin{document}

\begin{abstract}
Let $M_{l,m}$ be the total space of the $S^3$-bundle over $S^4$ classified by the element $l\sigma+m\rho\in{\pi_4(SO(4))}$, $l,m\in\mathbb Z$. In this paper we study the homotopy theory of gauge groups of  principal $G$-bundles over manifolds $M_{l,m}$ when $G$ is a simply connected simple compact Lie group such that $\pi_6(G)=0$. That is, $G$ is one of the following groups: $SU(n)$ $(n\geq4)$, $Sp(n)$ $(n\geq2)$, $Spin(n)$ $(n\geq5)$, $F_4$, $E_6$, $E_7$, $E_8$. If the integral homology of $M_{l,m}$ is torsion-free, we describe the homotopy type of the gauge groups over $M_{l,m}$ as products of recognisable spaces. For any manifold $M_{l,m}$ with non-torsion-free homology, we give a $p$-local homotopy decomposition, for a prime $p\geq 5$, of the loop space of the gauge groups.
\end {abstract}

\maketitle
\tableofcontents

\section{Introduction and main results}

Let $P_f\rightarrow X$ be a principal $G$-bundle over $X$ classified by a map $f:X\to BG$. The (unpointed) gauge group of the bundle, denoted $\mathcal G^f(X)$, is the group of its bundle automorphisms over $X$. That is, an element $\phi\in \mathcal G^f(X)$ is a $G$-equivariant automorphism of $P_f$ lying over the identity map on $X$. The subgroup of $\mathcal{G}^f(X)$ that fixes one fiber is called the pointed gauge group and it is denoted $\mathcal G_*^f(X)$. In this work we aim to classify, up to homotopy, the gauge groups of principal $G$-bundles over manifolds that arise as total spaces of $S^3$-bundles over $S^4$ for $G$ a simply connected simple compact Lie group.

The study of the topology of the gauge groups and their classifying spaces, when $G$ is a Lie group and $X$ is a compact low  dimensional manifold, has played a prominent role in the development of elementary particle theories in physics and the classification of 4-manifolds. Considerable attention has been paid in counting the number of homotopy types of gauge groups and their classifying spaces (see for instance \cite{Th1,Th2,Th3,Spf}). Crabb and Sutherland proved that if $X$ is connected and $G$ is a compact connected Lie group, the number of homotopy types of principal $G$-bundles over $X$ is finite \cite{CS}. In \cite{DT} Donaldson and Thomas introduced some ideas to extend the study of gauge theories to analogous situations in higher dimensions, where some special geometric structures over $X$ are required. It has been shown that certain compact 7-dimensional manifolds present the desired geometric properties. Moreover, the homotopy type of some of these manifolds has been described as a connected sum of total spaces of $S^3$-bundles over $S^4$ \cite{CHNP}.

 An $S^3$-bundle over $S^4$ is a 7-manifold $M$ with a projection map $\pi:M_{}\to S^4$, such that for all $x\in S^4$, there is a homeomorphism $\pi^{-1}(x)\cong S^3$. We can write $$S^3\overset{i}{\longrightarrow} M\overset{\pi}{\longrightarrow}S^4,$$
 where $i$ is the inclusion of the fibre. All the manifolds $M$ are compact, closed, orientable and 2-connected. The group $\pi_3(SO(4))\cong\mathbb{Z}\times\mathbb{Z}$ classifies  $S^3$-bundles over $S^4$ \cite{Stn2}, and the generators of this group are homomorphisms $\rho:S^3\rightarrow SO(4)$ and $\sigma:S^3\rightarrow SO(4)$ defined so that if $q,q'\in S^3$ then
\begin{equation*}
\rho(q)q'=q\cdot q'\cdot q^{-1},
\end{equation*}
\begin{equation*}
\sigma(q)q'=q\cdot q',
\end{equation*}
where $x\cdot y$ represents quaternionic multiplication. Let $M=M_{l,m}$ be the manifold that arises as the total space of the $S^3$-bundle over $S^4$ classified by an element $l\rho + m\sigma\in\pi_3(SO(4))$, where $l,m\in\mathbb Z$.
All the bundles $\pi:M_{l,m}\to S^4$  with $m=0$ admit cross sections. 
Sometimes the manifolds $M_{l,0}$ are referred to as twisted products and are denoted by $S^4\tilde\times_l S^3$. The   manifolds $M_{l,m}$ with $|m|=1$ are homotopy equivalent to $S^7$. If $|m|\geq2$, then $M_{l,m}$ has torsion in homology. 

We are interested in the homotopy theory of principal $G$-bundles over manifolds $M_{l,m}$ when the group $G$ is a simply connected simple compact Lie group. Given a pointed space $X$, we denote by $Prin_G(X)$ the set of isomorphism classes of principal $G$-bundles over $X$. 
 In Section \ref{s:classpring} it is showed that if $\pi_6(G)=0$ then $Prin_G(M_{l,m})=\mathbb{Z}_m$. Here $\mathbb Z_0=\mathbb Z$ and $\mathbb Z_1=0$. The simply connected simple compact Lie groups satisfying the condition $\pi_6(G)=0$ are the following: $SU(n)$ $(n\geq4)$, $Sp(n)$ $(n\geq2)$, $Spin(n)$ $(n\geq 5)$, $F_4$, $E_6$, $E_7$ and $E_8$. 

 Our first main result is to prove that the homotopy type of gauge groups over manifolds $M_{l,m}$ with torsion free homology depends on the homotopy type of gauge groups over $S^4$. It is well known that $Prin_G(S^4)=\mathbb{Z}$. Let $\mathcal G^k(S^4)$  be the unpointed gauge group over $S^4$ classified by~$k\in\mathbb{Z}$.  Given a map represented by $\xi_l\in \pi_6(S^3)\cong\mathbb{Z}_{12}$,  let $Y_{l}$ be its homotopy cofibre.  In Section 4 we prove the following theorem.
\begin{theorem} \label{t:UGGTF}
Let $G$ be a simply connected simple compact Lie group such that $\pi_6(G)=0$ and let $M_{l,0}$ be the total space of an $S^3$-bundle over $S^4$ with a cross section. Let $P_k\to M_{l,0}$ be a principal $G$-bundle classified by $k\in\mathbb{Z}$. There is a homotopy equivalence
 $$\mathcal{G}^k(M_{l,0})\simeq \mathcal{G}^k({S}^4)\times {\rm{Map}}_*(Y_l, G).$$
Moreover, if $l\equiv 0\pmod{12}$ there is a homotopy equivalence
$$\mathcal{G}^k(M_{l,0})\simeq \mathcal{G}^k({S}^4)\times\Omega^3G\times\Omega^7G.$$
\end{theorem}
Theorem \ref{t:UGGTF} implies that the determination of the homotopy type of $\mathcal{G}^k(M_{l,m})$ is reduced to determining that of $\mathcal{G}^k(S^4)$. These gauge groups have been computed for different $G$. For example, from \cite[Theorem 1.1]{Th4}  we obtain the following corollary.
Let $(n_1,n_2)$ be the greatest common divisor of $n_1$ and $n_2$.

\begin{cor}\label{t:cor}
Suppose $M$ is either $S^3\times S^4$ or any twisted product $S^3\tilde\times_l S^4$. Let $P_k\to M$ and $P_{k'}\to M$ be principal SU(5)-bundles. There is a homotopy equivalence $\mathcal{G}^k(M)\simeq\mathcal{G}^{k'}(M)$, if $(120,k)=(120,k')$, when localised rationally or at any prime $p.$ 
\end{cor}

The proof of Theorem \ref{t:UGGTF} relies on the splitting of the homotopy cofibre $C_{l,m}$ of the projection map. For the case of manifolds with torsion in homology, it is not clear if analogous splittings exist, however, we are able to obtain a splitting of $\Sigma C_{l,m}$. As such, the results in Theorem \ref{t:UGGT} are stated in terms of the loop spaces of the gauge groups rather than the gauge groups themselves.

The cofibration $S^n\overset{m}{\rightarrow} S^n\rightarrow P^{n+1}(m)$ induces a fibration $${\rm{Map}}_*(P^{n+1}(m),BG)\rightarrow{\rm{Map}}_*(S^n,BG)\overset{m^*}{\rightarrow}{\rm{Map}}_*(S^n,BG),$$ where $m_*$ is the $m$-$th$ power map. Let $\Omega^{n}BG\{m\}$ denote the space ${\rm{Map}}_*(P^{n+1}(m),BG)$. Let $v_p(m)$ be the $p$-adic valuation of $m$ at $p.$ If $M_{l,m}$ has torsion in homology we have the following result. 

\begin{theorem}\label{t:UGGT}
Let $G$ be a simply connected simple compact Lie group such that $\pi_6(G)=0$. Let $m> 1$ be an integer and $p\geq 5$ be a prime. Let $P_k\to M_{l,m}$ be a principal $G$-bundle classified by $k\in\mathbb{Z}_{m}$.   
There are p-local homotopy equivalences
\begin{enumerate}[(1)]
\item $\mathcal{G}^0(M_{l,m})\simeq_{(p)} \Omega^7G\times G,$ if $v_p(m)=0;$
\item 
$\Omega\mathcal{G}^{k}(M_{{l,{m}}{}})\simeq_{(p)} \Omega^8_0G_{}\times X_{k},$
if $v_p(m)\geq 1$, where there exists a homotopy fibration $$\Omega^4_0G\{m\}\to X_k\to \Omega G.$$ \end{enumerate}
Moreover, if $v_p(m)=r\geq1$ and $p^r|k$ then $X_k\simeq_{(p)} \Omega G\times \Omega^4_0 G\{m\}$.

\end{theorem}
We conclude this paper with a classification of gauge groups of principal $G$-bundles over manifolds $M_{l,1}$, which  are homotopy equivalent to $S^7$. In Section \ref{s:homos7} we prove the following result.

 \begin{theorem}\label{t:S}
Let $G$ be a simply connected simple compact Lie group and let $P_k\to S^7$ and $P_{k'}\to S^7$ be principal $G$-bundles. Then
\begin{enumerate}[(1)]
\item  for $G=SU(2)\cong Sp(1)$ there is a homotopy equivalence $\mathcal G^k(S^7)\simeq\mathcal{G}^{k'}(S^7)$ if and only if $(3,k)=(3,k')$;
\item  for $G = G_2$, there is a homotopy equivalence $\mathcal G^k(S^7)\simeq\mathcal{G}^{k'}(S^7)$ when localised rationally or at any prime if and only if $(3,k)=(3,k')$;
\item  for $G=SU(3)$, there is a homotopy equivalence $\mathcal G^k(S^7)\simeq\mathcal{G}^{k'}(S^7)$ when localised rationally or at a prime $p\geq3$ if and only if $(3,k)=(3,k')$;
\item otherwise, the gauge group of the unique principal $G$-bundle decomposes as $${\mathcal{G}^0(S^7)\simeq \Omega^7G\times G}.$$
\end{enumerate}
\end{theorem}

\begin{remark}
We want to point out that part \textit{(1)} of Theorem \ref{t:S}  contrasts with the results given in  \cite[Proposition 2]{Spf}, where it is stated that integrally, if $G=S^3$ all gauge groups over $S^7$ are homotopy equivalent. Our results show that given two elements $k,k'\in[S^7,BSU(2)],$ it is not always true that $\mathcal{G}^k(S^7)\simeq \mathcal{G}^{k'}(S^7).$
\end{remark}

\begin{ack}
I would like to thank Shizuo Kaji for many valuable conversations and the anonymous referee for suggesting a reformulation of Theorem 1.4 and for making comments which have helped to improve the clarity of this paper. I would like to give special thanks to Stephen Theriault for his advice and encouragement during the development of this research. This research project was supported by the Mexican National Council for Science and Technology (CONACyT) through the scholarship 313812.
\end{ack}

 \section{Classification of principal $G$-bundles over $M_{l,m}$}
 \label{s:classpring}

All spaces  considered in this work have the homotopy type of $CW$-complexes with non-degenerate basepoints. For given spaces $X$ and $Y,$ let Map$(X,Y)$ and Map$_*(X,Y)$ be the spaces of unpointed and pointed maps from $X$ to $Y,$ respectively. We endow these spaces with the compact-open topology. The path components of the corresponding mapping spaces containing the map $f$ are denoted Map$^f(X,Y)$ and Map$^f_*(X,Y)$. 
 We denote by $\langle X,Y\rangle$ and $[X,Y]$ the sets of homotopy classes of unpointed and pointed maps from $X$ to $Y$, respectively. Given a map $f:X\rightarrow Y$, we denote its homotopy class by the same letter $f$. The finite cyclic group of $n$ elements is denoted $\mathbb{Z}_n$. The localisation of $\mathbb Z$ at a prime $p$ is denoted $\mathbb Z_{(p)}.$ 

Let $M_{l,m}$ be the total space of the $S^3$-bundle over $S^4$,
classified by the element $$l\rho+m\sigma\in\pi_3(SO(4))\cong\mathbb{Z}\times\mathbb{Z},$$ where $l,m\in\mathbb{Z}$.  
 First observe that since all manifolds $M_{l,m}$ are 2-connected, we can give $M_{l,m}$ a minimal cellular structure
 \begin{equation*}
S^3\cup_{\varphi'_{}} e^4\cup_{\varphi_{}}e^7,
\end{equation*}
where $\varphi'_{}$ and $\varphi_{}$ are the attaching maps of the 4-cell and the 7-cell, respectively.
There are homeomorphisms \cite{Stn2} $M_{l,m}\cong M_{-l,-m}$ and $M_{l,m}\cong M_{l+m,-m}$ so that from now on we will only consider the case $m\geq 0$. 
The 4-skeleton $M^4_{l,m}$ of $M_{l,m}$ is given by the pushout 
\begin{equation}
\xymatrix{
S^3\ar[r]\ar[d]_-{\varphi'}&D^4\ar[d]\\
S^3\ar[r]&S^3\cup_{\varphi'} D^4 \cong M_{l,m}^4
}
\end{equation}
where $\varphi'$ is degree $m$ map and $m$ is the integer associated to the classifying element $l\rho+m\sigma\in\pi_4(SO(4))$ \cite{JW1}. If $m= 0$, then $M^4_{l,m}\simeq S^3\vee S^4,$ and in this case we have 
\begin{equation*}
M_{l,0}\simeq (S^3\vee S^4)\cup_{\varphi}e^7,
\end{equation*}
for $\varphi\in\pi_6(S^3\vee S^4)$. All the sphere bundles $M_{l,0}\xrightarrow{\pi}S^4$ admit a cross section so that the exact sequences of the fibre bundles show that the homotopy and homology groups of the manifolds $M_{l,0}$ are isomorphic to those of the total space of the trivial bundle $S^3\times S^4$. 
James and Whitehead  showed that $M_{l,0}\simeq M_{l',0}$ 
if and only if $l\equiv \pm l' \pmod{12}$ \cite{JW1} .

If $m> 0$, then the 4-skeleton $M^4_{l,m}$ is the Moore space $P^4(m)$ which is the homotopy cofibre of the degree $m$ map and satisfies $\tilde H_i(P^4(m))\cong\mathbb{Z}_m$, if $i=3$, and $\tilde H_i(P^4(m))\cong0$ otherwise. Thus given $M_{l,m}$ and $M_{l,m'}$, if $m\neq m'$ then $\pi_3(M_{l,m})\ncong \pi_3(M_{l,m'})$,  and therefore these spaces are not homotopy equivalent.
A minimal cellular structure for $M_{l,m}$ is given by
\begin{equation*}
M_{l,m}\simeq P^4(m)\cup_{\varphi}e^7,
\end{equation*}
for some $\varphi\in\pi_6(P^4(m))$. In \cite{CE} Crowley and Escher classified the homotopy types of manifolds~$M_{l,m}$ for $m> 0$. They showed that there is an orientation preserving homotopy equivalence $M_{l,m}\simeq M_{l',m'}$ if and only if $m = m'$ and 
 $l'\equiv\alpha l \pmod{(m, 12)}$ where $\alpha^2\equiv 1 \pmod{(m,12)}$. Thus in the case $m=1$ we have $P^4(1)\simeq *$ and we have $$M_{l,1}\simeq S^7$$
 for all $l\in\mathbb{Z}$. 

We make use of the Serre spectral sequence to obtain information on the properties of the projection maps $\pi:M_{l,m}\to S^4$ .

 \begin{lemma}\label{spect}
The map $\pi^*:H^4(S^4)\to H^4(M_{l,m})$ is an isomorphism if $m=0,$ reduction modulo m if $m>0$ and, in particular, the constant map if $m=1.$
\end{lemma}
\begin{proof}
 Consider the Serre spectral sequence of the sphere bundle
$$\xymatrix{
S^3\ar[r]& M_{l,m}\ar[r]^{\pi} &S^4
}$$
which converges to $H^*(M_{l,m}),$ and let $y_3$ and $x_4$ be suitable generators of $H^3(S^3)\cong \mathbb{Z}$ and ${H^4(S^4)\cong\mathbb{Z}}$ respectively. Then the $E_2^{p,q}$ page in the spectral sequence has the following form.
\begin{center}
\begin{tabular}{c | c c c}
3&$y_3$&&$y_3x_4$\\
&&&\\
0&1&&$x_4$\\
\hline
&0&&4
\end{tabular}
\end{center}
Thus we have that $E_2^{p,q}=E_4^{p,q}=H^p(S^3)\otimes H^q(S^4)$, and for dimensional reasons there is at most one non-trivial differential, namely $d_4(y_3)=mx_4$. This implies the result.  
\end{proof}

In order to obtain a classification of the principal $G$-bundles over $M_{l,m}$, when $m\neq1$, it is necessary to obtain information on the homotopy cofibre of the inclusion of the bottom cell 
\begin{equation}\label{c:s3inc}
\xymatrix{
S^3\ar[r]^-{i}& M_{l,m}\ar[r]^-{q}& D_{l,m}.
}
\end{equation}

\begin{lemma}\label{quotient}
There is a homotopy equivalence $$D_{l,m}\simeq S^4\vee S^7$$ and the homotopy equivalence can be chosen so that the composite $$M_{l,m}\xrightarrow{q} D_{l,m}\xrightarrow{\simeq}S^4\vee S^7\xrightarrow{pinch} S^4$$
is homotopic to the projection $\pi:M_{l,m}\to S^4$.
\end{lemma}
\begin{proof}
Since $\pi_3(S^4)\cong0$, the projection $M_{l,m}\xrightarrow{\pi}S^4$ has a homotopy extension $\tilde\pi$
\begin{equation}\label{C2}
\xymatrix{
S^3\ar[r]^-{i}&M_{l,m}\ar[r]^-q\ar[d]_-{\pi}&D_{l,m}\ar@{.>}[ld]^-{\tilde \pi}\\
&S^4.
}
\end{equation}
Observe that the cofibre $D_{l,m}$ can be built as a $CW$-complex with one 7-cell attached to a 4-sphere. Thus $D_{l,m}$ fits into the following cofibration sequence
\begin{equation}\label{cwc}
\xymatrix{
S^6\ar[r]^{\theta}& S^4\ar[r]^-{g}&D_{l,m}\ar[r]^-{b}&S^7,
}
\end{equation}
where $g$ is the inclusion, $\theta\in \pi_6(S^4)\cong\mathbb{Z}_2$ and $b$ is the connecting map. Let $M_{l,m}$ be a manifold with $m=0.$ Then the map $\pi:M_{l,0}\to S^4$ has a cross section $S^4\rightarrow M_{l,0}$.
Therefore, by the homotopy commutativity of \eqref{C2}, the map $\tilde\pi$ also has a right homotopy inverse. Now suppose $m>1$. The map $S^4\overset{g}{\rightarrow} D_{l,m}$ is the inclusion of the bottom cell and induces an isomorphism $g^*:H^4(D_{l,m})\xrightarrow{\cong} H^4(S^4)\cong\mathbb{Z}.$ Consider the commutative diagram:
\begin{equation}\label{dcch}
\xymatrix{
H^4(M_{l,m})&H^4(D_{l,m})\ar[l]_-{q^*}.\\
H^4(S^4)\ar[u]^-{\pi^{*}}\ar[ur]_-{\tilde \pi^{*}}&
}
\end{equation}
 By Lemma \ref{spect}, $\pi^{*}$ is reduction mod $m$. From \eqref{dcch} we obtain the following composite
$$\xymatrix{
\pi^*: \mathbb{Z}\ar[r]^-{\tilde\pi^*}&\mathbb{Z}\ar[r]^-{q^{*}}&\mathbb{Z}_{m},
}$$ which is reduction mod $m$. Thus $\tilde\pi^*=\pm1\pmod{m}.$  
Consider the homotopy commutative diagram 
\begin{equation}\label{d:P-S}
\xymatrix{
S^3\ar[r]^-{m}\ar[d]_-{\xi}&S^3\ar[r]\ar@{=}[d]&P^4(m)\ar[r]^-{q'}\ar[d]_-{j}& S^4\ar[d]^-{\pi'}\\
\Omega S^4\ar[r]^\delta&S^3\ar[r]&M_{l,m}\ar[r]^-\pi & S^4
}
\end{equation}
where the top row is a cofibration sequence and the bottom row is a fibration sequence. The connecting map $\delta$ induces multiplication by $m$ in cohomology. From the left square we obtain that $\xi$ is the inclusion of the bottom cell. By the Peterson-Stein formula the adjoint of the map $\xi$ is homotopic to $\pi'.$ Therefore  $\pi'$ is a homotopy equivalence. 
 This implies that $\tilde\pi$ has a right homotopy inverse. Therefore for all $m\neq 1$, the composite $S^4\xrightarrow{g} D_{l,m}\overset{\tilde\pi}{\to}S^4$ is a homotopy equivalence. There is a coaction $\sigma:D_{l,m}\to D_{l,m}\vee S^7$ such that the composite
$D_{l,m}\xrightarrow{\sigma} D_{l,m}\vee S^7\xrightarrow{pinch}S^7$
is homotopic to the connecting map $b$ in \eqref{cwc}, and the composite
$D_{l,m}\xrightarrow{\sigma} D_{l,m}\vee S^7\xrightarrow{pinch}D_{l,m}$ is a homotopy equivalence. Since $\tilde\pi$ has a right homotopy inverse, the composite
\begin{equation}
\theta:D_{l,m}\xrightarrow{\sigma}D_{l,m}\vee S^7\xrightarrow{\tilde\pi\vee\mathbbm 1}S^4\vee S^7
\end{equation}
is a homotopy equivalence. Also the diagram
\begin{equation}
\xymatrix{
D_{l,m}\ar[r]^-\sigma\ar@{=}[rd]&D_{l,m}\vee S^7\ar[r]^-{\tilde\pi\vee\mathbbm1}\ar[d]^-{pinch}&S^4\vee S^7\ar[d]^-{pinch}\\
&D_{l,m}\ar[r]^-{\tilde\pi}&S^4
}
\end{equation}
homotopy commutes.
Therefore by \eqref{C2} the composite $M\xrightarrow{q}D_{l,m}\xrightarrow{\theta}S^4\vee S^7\xrightarrow{pinch}S^4$ is homotopic to $\pi$.

\end{proof}

Given a compact connected topological space $X$  and a topological group $G$, let $Prin_G(X)$ be the set of isomorphism classes of principal $G$-bundles over $X$. It is well-known that there is a one-to-one correspondence between $Prin_G(X)$ and $\langle X,BG\rangle$, where $BG$ is the classifying space of $G$. 
The evaluation fibration
\begin{equation*}
{\rm{Map}}_*(M_{l,m},BG)\rightarrow{\rm{Map}}(M_{l,m},BG)\xrightarrow{ev}BG
\end{equation*}
induces an exact sequence of homotopy sets
$$\pi_1(BG)\xrightarrow{\partial} [M_{l,m},BG]\to\langle M_{l,m},BG\rangle\xrightarrow{ev^*}\pi_0(BG).$$
The induced map $ev^*$ is trivial as $BG$ is connected, and the coset space of $\partial(\pi_1(BG))$ coincides with the orbit space of the action of $\pi_1(BG)\cong\pi_0(G)$ on $[M_{l,m},BG]$. Since all groups $G$ considered in this work are connected, this action is trivial, which implies that there is a bijection between $[M_{l,m},BG]$ and  $\langle M_{l,m},BG\rangle$. We compute the sets $[M_{l,m},BG]$ for those manifolds with $m\neq1$.  We restrict to the case when $\pi_6(G)=0$, that is, when $G$ is one of the following groups: $SU(n)$ $(n\geq4)$, $Sp(n)$ $(n\geq2)$, $F_4$, $E_6$, $E_7$ or $E_8$.

\begin{proposition}\label{bundles}
Let $G$ be a simply connected simple compact Lie group such that $\pi_6(G)\cong 0$. \begin{enumerate}[(1)]
\item If $m = 0$ then $[M_{l,m},BG]=\mathbb{Z}$;
\item if $m \geq 2$ then $[M_{l,m},BG]=\mathbb{Z}_m$.
\end{enumerate} 
Moreover, the projection $M_{l,m}\overset{\pi}{\to} S^4$ induces a map $$\pi^*:[S^4,BG]\to [M_{l,m},BG]$$
which is a bijection if $m=0$ and a surjection if $m\geq2.$
\end{proposition}

\begin{proof}
For any simply connected simple compact Lie group $G$ there are isomorphisms \cite{Jam2} 
\begin{equation}\label{isoLG}
\pi_3(BG)\cong\pi_2(G)\cong0.
\end{equation}
 Let $f:M_{l,m}\rightarrow BG$ be a map. By \eqref{isoLG} the composite $S^3\hookrightarrow M_{l,m}\overset{f}{\rightarrow} BG$ is nullhomotopic. Using Lemma \ref{quotient} there is a homotopy commutative diagram 
 \begin{equation}\label{iso}
\xymatrix{
S^3\ar[r]^-{i}&M_{l,m}\ar[r]^-{a}\ar[d]_{f}& S^4\vee S^7\ar[r]^-{\delta}\ar@{.>}[dl]^-{\tilde f}&S^4\ar[r]^-{\Sigma i}&\Sigma M_{l,m}\\
&BG&&
}
\end{equation}
where $a$ is the composite $M_{l,m}\xrightarrow{q}D_{l,m}\xrightarrow{\simeq}S^4\vee S^7$, the top row is a cofibration sequence and $\tilde f:S^4\vee S^7\rightarrow BG$ is an extension of $f$. Therefore, applying the functor $[-,BG]$ to the cofibration sequence in \eqref{iso} we obtain an exact sequence of homotopy sets
 \begin{equation}\label{es:hs}
\xymatrix{ 
[S^4,BG] \ar[r]^-{\delta^{*}}\ar@{=}[d]& [S^4\vee S^7,BG]\ar[r]^-{a^{*}}\ar@{=}[d]&[M_{l,m},BG]\ar[r]^-{}& 0\\
\mathbb{Z}&\mathbb{Z}.
}
\end{equation}  Let $\psi':S^4\vee S^7\to S^4\vee S^7\vee S^4$ be the coaction of $S^4$ on $S^4\vee S^7$ in the cofibration sequence in \eqref{iso}. The homotopy set $[M_{l,m},BG]$ might not be a group. Therefore, we will use the action $$(\psi')^*:[S^4,BG]\times[S^4\vee S^7,BG]\to [S^4\vee S^7,BG],$$ induced by $\psi'$, to compute $[M_{l,m},BG].$

Let $M^4_{l,m}$ be the 4-skeleton of $M_{l,m}$. Then $M_{l,m}^4\simeq S^3\vee S^4$ or  $M_{l,m}^4\simeq P^4(m)$. In either case, $M^4_{l,m}$ is a co-$H$-space. From the exact sequence induced by the attaching map of the 4-cell,
\begin{equation}\label{c:4sk}
\xymatrix{
S^3\ar[r]^{m}&S^3\ar[r]&M_{l,m}^4\ar[r]&S^4,}
\end{equation}
we obtain an exact sequence of groups
\begin{equation}\label{4-skcof}
\xymatrix{
[S^4,BG]\ar[r]^-{m^*}&[S^4,BG]\ar[r]&[M^4_{l,m},BG]\ar[r]&0}
\end{equation}
where $\pi_4(BG)\cong\pi_3(G)\cong\mathbb{Z}$ and $m^*:\mathbb{Z}\to\mathbb{Z}$ is multiplication by $m$. The  coaction $\psi =S^4\to S^4\vee S^4$ associated to the cofibration \eqref{c:4sk} induces an action of homotopy sets
\begin{equation*}
\psi^*:[S^4,BG]\times [S^4,BG]\to [S^4,BG].
\end{equation*}
Exactness of \eqref{4-skcof} implies that $[M^4_{l,m},BG]=\mathbb{Z}_m$. By construction, the orbits under the action $\psi^*$ are equal to the cosets of the image of $m^*$. 
The map $S^3\xrightarrow i M_{l,m}$ factors through the 4-skeleton $M^4_{l,m}$. Therefore we have a homotopy cofibration diagram
\begin{equation}\label{d:coactions}
\xymatrix{
S^3\ar[r]\ar@{=}[d]&M^4_{l,m}\ar[r]\ar[d]&S^4\ar[r]^-{m}\ar[d]_-{i_1}& S^4\ar@{=}[d]\\
S^3\ar[r]^-i&M_{l,m}\ar[r]^-{a}&S^4\vee S^7\ar[r]^-\delta & S^4
}
\end{equation}
where $i_1:S^4\to S^4\vee S^7$ is the inclusion of the first factor into the wedge.  From \eqref{d:coactions} 
we obtain a homotopy commutative diagram as follows
\begin{equation}\label{d:coactions2}
\xymatrix{
S^4\ar[r]^-{\psi}\ar[d]_-{i_1}& S^4\vee S^4\ar[d]^-{i_1\vee\mathbbm{1}}\\
S^4\vee S^7\ar[r]^-{\psi'}& S^4\vee S^7\vee S^4.
}
\end{equation}
Applying the functor $[-,BG]$ we obtain a commutative diagram of homotopy groups
\begin{equation}\label{d:coactions3}
\xymatrix{
\pi_4(BG)\times \pi_7(BG)\times\pi_4(BG)\ar[r]^-{(\psi')^*}\ar[d]_-{i_1^*\times\mathbbm{1}}& \pi_4(BG)\times\pi_7(BG)\ar[d]^-{i_1^*}\\
 \pi_4(BG)\times\pi_4(BG)\ar[r]^-{\psi^*}& \pi_4(BG).\\
}
\end{equation}
Now assume $\pi_7(BG)\cong\pi_6(G)\cong0$. The vertical arrows in \eqref{d:coactions3} are isomorphisms implying that $(\psi')^*=\psi^*$. Since $(\psi')^*=\psi^*$, $[M_{l,m},BG]=\mathbb{Z}$ if $m=0,$ and $[M_{l,m},BG]=\mathbb{Z}_m$ if $m>1.$

Finally, we analyse the induced map $\pi^{*}:[M_{l,m},BG]\to [S^4,BG].$ By Lemma \ref{quotient} the composite $$M_{l,m}\xrightarrow{q} D_{l,m}\xrightarrow{\simeq}S^4\vee S^7\xrightarrow{pinch} S^4$$
is homotopic to the projection $\pi:M_{l,m}\to S^4$.
 Consider the commutative diagram
\begin{equation}\label{d:tri}
\xymatrix{
[M_{l,m},BG]&[S^4,BG]\times[S^7,BG]\ar[l]_-{a^*} \\
[S^4,BG]\ar[ur]_-{p_1}\ar[u]^-{\pi^{*}}
}
\end{equation}
where the map $p_1$ is the projection onto the first factor. From \eqref{es:hs} we get that $a^*$ is an isomorphism if $m=0$ and a surjection if $m>1$. Therefore, by the commutativity of \eqref{d:tri}, the induced map ${\pi^*:[M_{l,m},BG]\to[S^4,BG]}$ is an isomorphism if $m=0$, and a surjection if $m>1.$ 
\end{proof}

    Now let $m=1$. From the homotopy classification of the manifolds $M_{l,m}$, the spaces $M_{l,1}$ are homotopy equivalent to $S^7.$ Therefore we obtain $$[M_{l,1},BG]=[S^7,BG]=\pi_7(BG)\cong\pi_6(G).$$ From this and Proposition \ref{bundles} we obtain a classification of principal $G$-bundles over manifolds $M_{l,m}$ that holds for most of the simply connected simple compact Lie groups. Using the notation $\mathbb Z_0=\mathbb Z$ and $\mathbb Z_1=0$ we now state our result for the classification of principal $G$-bundles over $M_{l,m}.$

\begin{cor}\label{c:bundles}
Let $G$ be a simply connected simple compact Lie group such that $\pi_6(G)= 0$. There is a one-to-one correspondence  $$Prin_G(M_{l,m})\xrightarrow{1-1}\mathbb{Z}_m.$$ \qed
\end{cor}

 \section{Homotopy types of $\Sigma M_{l,m}$}
 \label{s:suspension}
  
In this section we discuss the homotopy types of $\Sigma M_{l,m},$ the suspensions of the manifolds $M_{l,m}$. The description of $\Sigma M_{l,m}$ will be needed later to obtain homotopy decompositions of the gauge groups. 

We start with a general result regarding the suspension of total spaces of $S^{n-1}$-bundles over $S^n$ that have cross sections. Let $X^n$ be the $n$-skeleton of $X.$ Let $1_X$ be the identity map on $X$.

 \begin{lemma}\label{highbun}    
Let $\pi:X\to S^n$ be
an $S^{n-1}$-bundle over $S^{n}$, $n\geq 3$, with a cross section.  Then $X^n \simeq S^{n-1} \vee S^n$ and there is a homotopy equivalence
\begin{equation*}
\Sigma X\simeq \Sigma Y\vee S^{n+1},
 \end{equation*}
 where $Y$ is the homotopy cofibre of the composite $S^{2n-2}\xrightarrow{\varphi} S^{n-1}\vee S^n\xrightarrow{pinch} S^{n-1}$. Here the map ${\varphi: S^{2n-2} \to S^{n-1} \vee S^n}$ is the attaching map of the top cell of $X$.
\end{lemma}
\begin{proof} 
The topological space $X$ is homotopy equivalent to a $CW$-complex with the following cellular structure
$$S^{n-1}\cup e^n\cup_\varphi e^{2n-1},$$
where $\varphi$ is the attaching map of the top cell. 
There is a homotopy commutative diagram
\begin{equation}\label{fibcofib}
\xymatrix{
S^{n-1} \ar@{^{(}->}[r] \ar@{=}[d] & X^n \ar[r]^-q \ar@{^{(}->}[d] & S^{n} \ar@{=}[d] \\
S^{n-1} \ar@{^{(}->}[r] & X \ar[r]^-{\pi} & S^{n}
}\end{equation}
where the top row is the cofibration sequence induced by the inclusion of the bottom cell into the $n$-skeleton, the bottom row is the fibration sequence of the sphere bundle, and $q$ is the quotient map. Since $X^n=X^{2n-2}$ by connectivity, the map $q$ also has a right homotopy inverse, implying that there is a homotopy equivalence $X^n\simeq S^{n-1}\vee S^n.$

Now consider the cofibration sequence induced by the attaching map $S^{2n-2}\overset{\varphi}{\longrightarrow}S^n\vee S^{n-1}$:
\begin{equation}\label{high}
\xymatrix{
S^{2n-2}\ar[r]^-{\varphi}&S^{n-1}\vee S^{n}\ar[r]^-{i}&X\ar[r]^-{\rho}&S^{2n-1}\ar[r]^{\Sigma \varphi\quad}& S^{n}\vee S^{n+1}\ar[r]^-{\Sigma i}&\Sigma X,
}
\end{equation}
where $i$ is the inclusion and $\rho$ is the pinch map to the $(2n-1)$-cell.
By the Hilton-Milnor Theorem \cite{HMT1,HMT2}
there is an isomorphism $$\pi_{2n-2}(S^{n-1}\vee S^{n})\cong \pi_{2n-2}(S^{2n-2})\times\pi_{2n-2}(S^n)\times \pi_{2n-2}(S^{n-1}).$$
Thus $\varphi \in \pi_{2n-2}(S^{n-1}\vee S^n)$ can be expressed as
\begin{equation}\label{attaching}
\varphi=t[1_{S^{n-1}},1_{S^n}]+\underline\alpha+\underline\beta
\end{equation}
Here the Whitehead product $[1_{S^{n-1}},1_{S^n}]$ factors through a generator of $\pi_{2n-2}(S^{2n-2})$ and $t \in \mathbb{Z}$; for any $\alpha \in \pi_{2n-2}(S^{n-1})$ and $\beta \in \pi_{2n-2}(S^{n})$, let $\underline{\alpha}$ and $\underline{\beta}$ be the elements of $\pi_{2n-2}(S^{n-1}\vee S^n)$ which are represented by the maps $$\underline{\alpha}: S^{2n-2}\overset{\alpha}{\rightarrow} S^{n-1}\hookrightarrow S^{n-1}\vee S^n$$ and $$\underline{\beta}: S^{2n-2}\overset{\beta}{\rightarrow} S^n\hookrightarrow S^{n-1}\vee S^n.$$ 
Consider the diagram 
\begin{equation*}
\begin{gathered}
\xymatrix{
S^{2n-2}\ar[r]^-{\varphi}\ar[dr]_{\beta}&S^n\vee S^{n-1}\ar[r]^-{i}\ar[d]^{p_1}&X\ar[d]^{\pi}\\
&S^n\ar@{=}[r]&S^n&
}
\end{gathered}
\end{equation*}
The triangle homotopy commutes by definition of $\varphi$ and $\underline\beta,$ and the square homotopy commutes by the commutativity of right square in \eqref{fibcofib}. Thus $\beta\simeq\pi\circ i\circ\varphi$, but $i\circ\varphi$ is nullhomotopic since $i$ and $\varphi$ are consecutive maps in a cofibration. Hence $\beta$ is nullhomotopic and therefore so is $\underline\beta$.
Hence \eqref{attaching} is reduced to $$\varphi=t[1_{S^{n-1}},1_{S^n}]+\underline\alpha.$$

After suspension we have $\Sigma\varphi=\Sigma\underline\alpha$ since $\Sigma[1_{S^{n-1}},1_{S^n}]\simeq *$. Let $Y$ be the homotopy cofibre of the map $\alpha: S^{2n-2}\rightarrow S^{n-1}$. Thus if $\Sigma\alpha\simeq{*}$  then $\Sigma\varphi\simeq *$. Therefore  the map $\Sigma i$ in \eqref{high} has a  left homotopy inverse, and $\Sigma X\simeq S^{2n}\vee S^{n}\vee S^{n+1}$. If instead $\Sigma\alpha$ is not nullhomotopic, then $\Sigma\varphi\not\simeq *$.  
Consider the following part of the homotopy cofibration sequence \eqref{high}
\begin{equation*}
\xymatrix{
S^{2n-1}\ar[r]^{\Sigma \varphi\quad}& S^{n+1}\vee S^{n}\ar[r]^-{\Sigma i}& \Sigma X.
}
\end{equation*}
Thus $\Sigma\varphi=\Sigma\underline\alpha=j\circ\Sigma\alpha$, where $j:S^{n}\longrightarrow S^{n}\vee S^{n+1}$ is the inclusion into the wedge. Therefore $~{\Sigma X\simeq \Sigma Y\vee S^{n+1}}$, where $Y$ is defined by the cofibration sequence  $$S^{2n-2}\overset{\alpha}{\longrightarrow} S^{n-1}\longrightarrow Y,$$ for $\alpha\in \pi_{2n-2}(S^{n-1})$.
\end{proof}
 Let $S^6\xrightarrow{\varphi} S^3\vee S^4$ be the attaching map of the top cell of the manifold $M_{l,0}$.
\begin{proposition}\label{p:susp}
There is a homotopy equivalence
$$\Sigma M_{l,0}\simeq\Sigma Y_{l}\vee S^5,$$ where $Y_{l}$ is the homotopy cofibre of the composite $S^{6}\xrightarrow{\varphi} S^{3}\vee S^4\xrightarrow{p_1} S^{3}$. Further, if $l'\equiv \pm l\pmod {12}$ there is a homotopy equivalence $\Sigma M_{l,0}\simeq \Sigma M_{l',0}.$ In particular $$\Sigma M_{l,0}\simeq S^8\vee S^4\vee S^5,$$
if $l\equiv 0\pmod{12}$.
\end{proposition}
\begin{proof}
 There is a cofibration sequence
 \begin{equation}\label{c:Mc}
 \xymatrix{
 S^6\ar[r]^-{\varphi}&S^3\vee S^4\ar[r]&M_{l,0}.
 }
 \end{equation}
where $\varphi$ is the attaching map of the top cell. We can write the attaching map $\varphi$ as $$\varphi=[\iota_3,\iota_4]+t_l\underline\nu'.$$ Here $[\iota_3,\iota_4]$ is the Whitehead product of the identiy maps of $S^3$ and $S^4$ and the map $t_l\underline\nu'$ is the composite $S^6\xrightarrow{t_l\nu'}S^3\hookrightarrow S^3\vee S^4$, where $\nu'$ is a generator of $\pi_6(S^3)\cong\mathbb Z_{12}$ \cite{Tod} and $t_l\in\mathbb Z_{12}$. Since the map $\pi:M_{l,0}\to S^4$ has a section, by Lemma \ref{highbun} there is a homotopy equivalence 
\begin{equation}\label{eq:susp}
\Sigma M_{l,0}\simeq\Sigma Y_{l,0}\vee S^5,
\end{equation}
 where $\Sigma Y_{l,0}$ is the homotopy cofibre of the composite $S^{7}\xrightarrow{\Sigma\varphi} S^{4}\vee S^5\xrightarrow{pinch} S^{4}$. The map $\Sigma\varphi$ is homotopic to the composite $S^7\xrightarrow{t_l\Sigma\nu'}S^4\hookrightarrow S^4\vee S^5$, where the element $\Sigma\nu'$ generates a subgroup of order 12 in $\pi_7(S^4)$ \cite{Tod}.

Set $Y_l:=Y_{l,0}$. The $J$-homomorphism $J:\pi_3(SO(3))\to \pi_6(S^3)$ which send, $l$ to $t_l$, is an epimorphism \cite{JW1}. Observe that two spaces $\Sigma Y_l$, $\Sigma Y_{l'}$ are homotopy equivalent if and only if there is a homotopy equivalence $\theta:S^4\to S^4$ such that $$\theta^*(t_l\Sigma\nu')=t_{l'}\Sigma\nu',$$
 where $\theta^*$ is the automorphism of $\pi_7(S^4)$ induced by $\theta.$ Any self-equivalence of $S^4$ is homotopic to $\pm1_{S^4}.$ Since, $t_l, t_{l'}\in\mathbb Z_{12}$, we have that $\Sigma Y_l\simeq \Sigma Y_{l'}$ if and only if $l'\equiv \pm l\pmod{12}.$ 
Thus if $l'\equiv \pm l'\pmod{12}$ then $\Sigma M_{l,0}\simeq\Sigma M_{l',0}$. In particular $M_{0,0}=S^3\times S^4$ and $\Sigma (S^4\times S^3)\simeq S^8\vee S^5\vee S^4$. Identifying summands in \eqref{eq:susp} we have that $Y_{0,0}\simeq S^8\vee S^4$ and therefore there is a homotopy equivalence 
 $\Sigma M_{l,0}\simeq S^8\vee S^5\vee S^4$ if $l\equiv 0\pmod{12}.$
\end{proof}

In order to obtain results on the gauge groups over manifolds $M_{l,m}$ with torsion in homology we will require $p$-localisations of nilpotent spaces in the sense of \cite{HMR}. A connected $CW$-complex is nilpotent if $\pi_1(X)$ is nilpotent and acts nilpotently on $\pi_n(X)$ for all $n\geq2$. In particular, if $X$ is simply connected then it is nilpotent. Let $X$ and $Y$ be connected $CW$-complexes. By \cite[Corollary 2.6]{HMR}, given a map $f:X\to Y$, if $Y$ is nilpotent and $X$ is finite then the path components ${\rm{Map}}^f(X,Y)$ and ${\rm{Map}}_*^f(X,Y)$ are nilpotent, and these mapping spaces admit $p$-localisations, $({\rm{Map}}^f(X,Y))_{(p)}$ and $({\rm{Map}}_*^f(X,Y))_{(p)}$, for $p$ a prime. Moreover $({\rm{Map}}^f_*(X,Y))_{(p)}\simeq {\rm{Map}}^f_*(X_{(p)},Y_{(p)})$. In our case, we will make use of $p$-localisations of simply connected finite spaces $X$ homotopy equivalent to $CW$-complexes. We will also need to localise the mapping spaces ${\rm{Map}}^f(X,BG)$ and ${\rm{Map}}_*^f(X,BG)$ which is possible to do  since we restrict to Lie groups $G$ with nilpotent classifying spaces $BG$. We define the $p$-localisation of the $n$-th loop space of a nilpotent space $Y$, at a prime $p$, as follows
$$ (\Omega^n Y)_{(p)}:=\Omega^n(Y_{(p)}).$$ 
Thus given a map $f:X\to BG$, where $X$ is a finite connected complex and $G$ is a simply connected simple compact Lie group, we define the $p$-localisation of $\Omega^n{\rm{Map}}^f(X,Y)$ as follows $$(\Omega^n {\rm{Map}}^f(X,BG))_{(p)}:=\Omega^n(({\rm{Map}}^f(X,BG))_{(p)}).$$  To keep the notation simple, in the following discussions we will avoid using the subscript $(p)$ when referring to local spaces. We denote by $X\simeq_{(p)}Y$ a ($p$-local) homotopy equivalence between the $p$-local spaces $X$ and $Y$.

 
The cofibration sequence $S^n\xrightarrow m S^n\rightarrow P^{n+1}(m)$ induces a fibration sequence $${\rm{Map}}_*(P^{n+1}(m),BG)\rightarrow{\rm{Map}}_*(S^n,BG)\overset{m_*}{\rightarrow}{\rm{Map}}_*(S^n,BG),$$ where $m_*$ is the $m$-th power map. Let $\Omega^{n}BG\{m\}$ denote the space ${\rm{Map}}_*(P^{n+1}(m),BG)$. Let $v_p(m)$ be the $p$-adic valuation of $m$ at $p.$ 

\begin{proposition}\label{torman}
Let $M_{l,m}$ be the total space of an $S^{3}$-bundle over $S^{4}$ with $m\geq 2$. Localising at $p\geq 5$ there exists a $p$-local homotopy equivalence 
\begin{equation*}
\Sigma M_{l,m}\simeq_{(p)} P^5(p^r)\vee S^8,
\end{equation*} 
where $r=v_p(m).$ 

\end{proposition}
\begin{proof}
There exists a cofibration sequence
\begin{equation}\label{Sbun}
\xymatrix{
S^6\ar[r]^-{\varphi}&P^4(m)\ar[r]^{i}&M_{l,m}\ar[r]^{\rho}&S^7\ar[r]^{\Sigma \varphi\quad}& P^5(m)\ar[r]^{\Sigma i}& \Sigma M_{l,m},
}
\end{equation}
where $\varphi$ is the attaching map of the top cell, $i$ is the inclusion and $\rho$ is the pinch map. Now  suppose that all spaces are localised at a prime $p\geq 5$ with $r=v_p(m)$.
Consider the cofibration sequence
\begin{equation}\label{moore}
\xymatrix{
S^3\ar[r]^m&S^3\ar[r]^-q&P^4(m).
}
\end{equation}
We have two cases to analyse: $r=0$ and $r\geq 1$. If $r=0$ then the degree map $m$ is invertible in $\mathbb{Z}_{(p)}$, so the map $m$ is a homotopy equivalence in the cofibration sequence \eqref{moore}, and therefore $P^4(m)\simeq *$. From \eqref{Sbun} we can see that the attaching map $\varphi$ is nullhomotopic and therefore $M_{l,m}\simeq_{(p)} S^7$. Moreover, we can write $P^5(1)=P^5(p^0)\simeq *$.  Hence there is a homotopy equivalence  $\Sigma M_{l,m}\simeq_{(p)}P^5(p^r)\vee S^8$ for $r=0$.

If $r\geq 1$ then the degree map $m$ is not invertible. Localising at $p$ we obtain
$P^4(m)\simeq_{(p)}P^4(p^r)$. In \cite{Sas} Sasao computed the homotopy group $\pi_6(P^4(m))$. He showed that integrally $$\pi_6(P^4(m))\cong \begin{cases} \mathbb{Z}_{(m,12)} \oplus \mathbb{Z}_m & \text{if } v_2(m) = 0, \\ \mathbb{Z}_{(m,12)/2} \oplus \mathbb{Z}_{2m} \oplus \mathbb{Z}_2 & \text{if } 1 \leq v_2(m) \leq 2, \\ \mathbb{Z}_{(m,12)} \oplus \mathbb{Z}_m \oplus \mathbb{Z}_2 & \text{if } v_2(m) \geq 3. \end{cases}$$ In all above cases, localising at $p\geq 5$ we obtain $$\pi_6(P^4(m))\cong \mathbb{Z}_{p^r}.$$
We give an alternative construction of a $p$-local generator of $\pi_6(P^4(m))$ to that given by Sasao.
Let ${\bar \sigma\in\pi_6(P^4(m))\cong \mathbb{Z}_{p^r}}$ be a generator. We can write the attaching map of the top cell as $\varphi=t\cdot\bar\sigma$ with $t\in\mathbb{Z}_{p^r}$. Notice that if $\Sigma\varphi\simeq* $ then $\Sigma i$ has a left homotopy inverse, implying $\Sigma M_{l,m}\simeq P^4(p^r)\vee S^8$. We claim that the generator $\bar\sigma$ suspends trivially.
 
Let $\nu:P^4(p^r)\rightarrow P^4(p^r)$ be the identity map. Since $\nu$ is a suspension there is a  Whitehead product ${[\nu,\nu]}:\Sigma P^3(p^r)\wedge P^3(p^r)\rightarrow P^4(p^r)$. There is a homotopy equivalence \cite{CMN}$$\Sigma P^3(p^r)\wedge P^3(p^r)\simeq_{(p)} P^7(p^r)\vee P^6(p^r).$$ This homotopy equivalence precomposed with the inclusion of $P^7(p^r)$ into the wedge determines a map $\widehat{[\nu,\nu]}:P^7(p^r)\rightarrow P^4(p^r)$. By a result of Cohen, Moore and Neisendorfer \cite{CMN}, there is a $p$-local homotopy equivalence
\begin{equation}\label{loopmoore}
\phi:S^3\{p^r\}\times\Omega \mathcal{A}\rightarrow \Omega P^4(p^r)  
\end{equation}
where $\mathcal{A}=\bigvee_{k=0}^\infty P^{4+2k+3}(p^r)$ and $S^n\{p^r\}$ denotes the homotopy fibre of the degree map $p^r:S^n\to S^n$. This homotopy equivalence is constructed so that the restriction of $\phi$ to $\Omega P^7(p^r)$ is homotopic to $\Omega\widehat{[\nu,\nu]}$.

Using \eqref{loopmoore} we get
\begin{equation*}
\pi_6(P^4(p^r))\cong\pi_5(\Omega P^4(p^r))\cong\pi_5(S^3\{p^r\})\oplus\pi_5(\Omega\mathcal{A}).
\end{equation*}
Notice that there is a homotopy fibration given by 
\begin{equation*}
\Omega S^3\longrightarrow S^3\{p^r\}\longrightarrow S^3.
\end{equation*}
As 2 and 3 are inverted we have $\pi_5(S^3)=0$ and $\pi_5(\Omega S^3)\cong\pi_6(S^3)=0$ and therefore ${\pi_5(S^3\{p^r\})=0}$.  Now $\pi_5(\Omega\mathcal{A})\cong\pi_6(\mathcal{A})\cong\pi_6(P^7(p^r))\cong\mathbb{Z}_{p^r}$, where the last two isomorphisms are given by the high connectivity of the factors in the wedge defining $\mathcal{A}$ and the Hurewicz isomorphism, respectively. Thus a generator $\bar\sigma$ of $\pi_6(P^4(p^r))$ is represented by the map
\begin{equation*}
\bar\sigma:S^6\hookrightarrow P^7(p^r)\overset{\widehat{[\nu,\nu]}}{\longrightarrow}P^4(p^r).
\end{equation*}
Since $\widehat{[\nu,\nu]}$ factors through the Whitehead product ${[\nu,\nu]}$, which suspends trivially, we obtain $\Sigma\bar\sigma\simeq*$, as claimed. 
\end{proof}
 
\section{Homotopy decompositions of gauge groups}
\label{s:torfree}

In this section we give homotopy decompositions of the gauge groups over $M_{l,m}$ for which $m\neq1$. We split our results in two cases: manifolds $M_{l,m}$ with torsion-free homology and manifolds $M_{l,m}$  with non-torsion-free homology.

Let $P_f\rightarrow X$ be a principal $G$-bundle with classifying map $f:X\to BG$. Recall that the (unpointed) gauge group of the bundle, denoted $\mathcal G^f(X)$, is the group of its bundle automorphisms. That is, an element $\phi\in \mathcal G^f(X)$ is a $G$-equivariant automorphism of $P_f$ lying over the identity map on $X$. The subgroup of $\mathcal{G}_f(X)$ that fixes one fiber is the pointed gauge group and it is denoted $\mathcal G_*^f(X)$. Let $B\mathcal G^f(X)$ be the classifying space of $\mathcal G^f(X)$. 
From \cite{Got} or  \cite{AB} there are homotopy equivalences 
\begin{equation}\label{Atiyah}
B\mathcal{G}^f(X)\simeq{\rm{Map}}^f(X,BG),
\end{equation}
\begin{equation}\label{Atiyahun}
B\mathcal{G}^f_*(X)\simeq{\rm{Map}}^f_*(X,BG),
\end{equation}
and after looping these homotopy equivalences we obtain 
\begin{equation}\label{Atiyah22}
\mathcal{G}^f(X)\simeq\Omega{\rm{Map}}^f(X,BG),
\end{equation}
\begin{equation}\label{Atiyahunn}
\mathcal{G}^f_*(X)\simeq\Omega{\rm{Map}}^f_*(X,BG).
\end{equation}

 We will make use of the equivalences \eqref{Atiyah}-\eqref{Atiyahunn} to obtain homotopy decompositions of the gauge groups. Following our discussion on localisation of nilpotent spaces (see Section 3), we define the $p$-localisation of the gauge groups $\mathcal G^f(X)$ and $\mathcal G^f_*(X)$ as  $(\mathcal G^f(X))_{(p)}:=\Omega((B\mathcal G^f(X))_{(p)})$ and $(\mathcal G^f_*(X))_{(p)}:=\Omega((B\mathcal G^f_*(X))_{(p)}).$

We specialise to gauge groups of principal $G$-bundles over $M_{l,m}$, where $M_{l,m}$ is the total space of a sphere bundle with classifying map $l\rho+m\sigma\in\pi_3(SO(4))$, $l,m\in\mathbb Z$, and $G$ is a simply connected simple compact Lie group with $\pi_6(G)\cong0$. By Corollary \ref{c:bundles} there are set isomorphisms $$Prin_G(M_{l,m})\xrightarrow{\cong}[M_{l,m},BG]\xrightarrow{\cong}\mathbb Z_m.$$ 

Let $\mathcal G_k(M_{l,m})$ be the gauge group of the principal $G$-bundle over $M_{l,m}$ classified by  $k\in\mathbb Z_m=[M_{l,m},BG]$. By Proposition \ref{bundles}, the projection $\pi: M_{l,m} \to S^4$ induces a map $\pi^*:[S^4,BG]\to [M_{l,m},BG]$ which is a bijection if $m=0$ and a surjection if $m\geq 2$.

\subsection{Torsion-free case}
 Let $M_{l,m}$ be a manifold with torsion-free homology and $M_{l,m}\not\simeq S^7$. Thus in this case $m=0$ and the bundle $M_{l,0}\to S^4$ has a cross section.  The following lemma will be crucial to identify the spaces that appear in the homotopy decompositions of the gauge groups.
\begin{lemma}\label{l:diagram}
There is a homotopy commutative diagram of cofibrations
\begin{equation}\label{d:deltadiag}
\xymatrix{
{*}\ar[r]\ar[d]&S^7\ar@{=}[r]\ar@{^{(}->}[d]&S^7\ar[d]^{\gamma}\ar[r]&{*}\ar[d]\\
M_{l,0}\ar@{=}[d]\ar[r]^-{}&S^4\vee S^7\ar[r]^-{\delta}\ar[d]^{{pinch}}&S^4\ar[r]^{}\ar[d]^{}&\Sigma M_{l,0}\ar@{=}[d]\\
M_{l,0}\ar[r]^{\pi}&S^4\ar[r]^-{}&C_{l,0}\ar[r]^-{}&\Sigma M_{l,0}
}
\end{equation}
which defines the space $C_{l,0}$ and the maps $\gamma$ and $\delta$. Furthermore
there is a homotopy equivalence
$C_{l,0} \xrightarrow\simeq{\Sigma Y_{l}}.$
\end{lemma}
\begin{proof} 
In general, by Lemma \ref{quotient} given a manifold $M_{l,m}$ there is a homotopy commutative diagram
\begin{equation}\label{csq}
\xymatrix{
S^3\ar[r]^-{i}&M_{l,m}\ar[r]^-{}\ar[d]_-{\pi}&S^4\vee S^7\ar[ld]^-{pinch}\\
&S^4&
}
\end{equation} 
Now let $m=0$. Using \eqref{csq} we can generate a homotopy commutative diagram as the one stated in the proposition, where each column and row is a cofibration sequence. Here $\gamma\in\pi_7(S^4)\cong\mathbb Z\times\mathbb Z_{12}$ is a map making the middle upper square of diagram \eqref{d:deltadiag} commute, which defines the space $C_{l,0}$. From the exact sequence induced by the top row in \eqref{csq}
\begin{equation}\label{s:coho}
\xymatrix{
H^4(S^4)\ar[r]^-{\delta^*}\ar@{=}[d]&H^4(S^4\vee S^7)\ar[r]^-{q^*}\ar@{=}[d]&H^4(M_{l,m})\ar@{=}[d]\ar[r]&0\\
\mathbb{Z}&\mathbb{Z}&\mathbb{Z}_m
}
\end{equation}
we conclude that $\delta$ restricted to $S^4$ is the degree $m$ map so that $\mathbb Z_0=\mathbb Z$. Let $s:S^4\to M_{l,0}$ be a section of the projection map $\pi$. Since $\pi\circ s= 1_{S^4}$, then $\Sigma\pi\circ\Sigma s\simeq 1_{S^5}$. Therefore the map $$\psi:S^5\vee C_{l,0}\xrightarrow{\Sigma s+b} \Sigma M_{l,0},$$ where $b$ is the connecting map of the cofibration induced by the projection $\pi$, is a homotopy equivalence. By Proposition \ref{p:susp} there is a homotopy equivalence
\begin{equation}\label{eq:sus}
\theta:\Sigma M_{l,0}\xrightarrow\simeq S^5\vee\Sigma Y_l,
\end{equation}
where $Y_l$ is the homotopy cofibre of a map $t_l\Sigma\nu'\in \pi_7(S^4)$ and $t_l$ depends linearly on $l$. Let $h_l=t_l\Sigma\nu'$. The suspension of the attaching map generates a homotopy commutative diagram of cofibrations
\begin{equation}
\xymatrix{
{*}\ar[r]\ar[d]&S^5\ar@{=}[r]\ar[d]&S^5\ar[d]^-{\Sigma s}\\
S^7\ar[r]^-{\Sigma\varphi}\ar@{=}[d]&S^5\vee S^4\ar[r]\ar[d]^-{pinch}&\Sigma M_{l,0} \ar[d]^-{c}\\
S^7\ar[r]^-{h_{l}}&S^4\ar[r]^-{}&\Sigma Y_{l}
}
\end{equation}
which defines the map $c$.
The homotopy commutative diagram of cofibrations
\begin{equation}
\xymatrix{
{*}\ar[r]\ar[d]&C_{l,0}\ar@{=}[r]\ar[d]^-{b}&C_{l,0}\ar[d]^-{}\\
S^5\ar[r]^-{\Sigma s}\ar@{=}[d]&\Sigma M_{l,0}\ar[r]^-c\ar[d]^-{\Sigma\pi}&\Sigma Y_l \ar[d]^-{}\\
S^5\ar[r]^-{\simeq}&S^5\ar[r]^-{}&{*}
}
\end{equation}
shows that there is a homotopy equivalence $C_{l,0}\xrightarrow\simeq \Sigma Y_{l,0}$.
 \end{proof}

By Proposition \ref{bundles}, the projection map induces a bijection between path components 
\begin{equation}
\pi^*:[S^4,BG]=\pi_0({\rm{Map}}_*(S^4,BG)) \rightarrow [M_{l,m},BG]=\pi_0({\rm{Map}}_*(M_{l,m},BG)).
\end{equation}
Moreover, the projection map $\pi$ induces the following fibration sequences

$${\rm{Map}}_*(\Sigma Y_l,BG)\to {\rm{Map}}_*(S^4,BG){\xrightarrow{\pi^*}}{\rm{Map}}_*(M_{l,0},BG),$$ $$F^k_{l,0}\to {\rm{Map}}_*^k(S^4,BG){\xrightarrow{\pi^{*}_k}}{\rm{Map}}_*^k(M_{l,0},BG),$$ where $\pi_k^{*}$ is the restriction of $\pi^*$ to the $k$-th component and $F^k_{l,0}$ is the corresponding homotopy fibre. Using the bottom row in the commutative diagram of Lemma \ref{l:diagram} we obtain the following fibration sequence for $k=0$ 
$${\rm{Map}}_*(\Sigma Y_l,BG)\to {\rm{Map}}_*^0(S^4,BG){\xrightarrow{\pi^*_0}}{\rm{Map}}_*^0(M_{l,0},BG),$$
where we can identify ${\rm{Map}}_*(\Sigma Y_l,BG)\simeq {\rm{Map}}_*(Y_l,G).$

Next we state a general result on the homotopy types of the spaces $F^k_{l,0}.$
\begin{lemma}\label{hofibre}
Let $G$ be a simply connected simple compact Lie group with $\pi_6(G)=0$. Let $F^k_{l,0}$ be the homotopy fibre of 
$\pi^*_{k}:{\rm{Map}}_*^k(S^4,BG) \rightarrow {\rm{Map}}^k_*(M_{l,0},BG).$ 
There are homotopy equivalences $$F^k_{l,0}\simeq{\rm{Map}}_*(Y_{l},G), \text{ for all }k\in \mathbb{Z}.$$
\end{lemma}
\begin{proof}
The inclusion of the bottom cell into $M_{l,m}$ induces a fibration sequence 
\begin{equation}\label{f:mapincs3}
{\rm Map}_*(S^4\vee S^7,BG)\xrightarrow{q^*}{\rm Map}_*(M_{l,0},BG)\xrightarrow{i^*}{\rm Map}_*(S^3,BG)
\end{equation}
Applying the functor Map$_*(-,BG)$ to the diagram in Lemma \ref{l:diagram}, we can fit the fibration sequence \eqref{f:mapincs3} into a homotopy commutative diagram
\begin{equation}\label{d:fib1}
\xymatrix{
{\rm Map}_*(\Sigma Y_{l},BG)\ar[r] \ar[d] & {\rm Map}_*(S^4,BG) \ar[r]^-{\gamma^*} \ar[d]_-{(0+\gamma)^*} & {\rm Map}_*(S^7,BG) \ar@{=}[d] \\
{\rm Map}_*(S^4,BG) \ar[r]^-{i^*} \ar[d]_-{\pi^*} & {\rm Map}_*(S^4,BG) \times {\rm Map}_*(S^7,BG)\ar[r]^-{p_2^*} \ar[d]_-{q^*} & {\rm Map}_*(S^7,BG) \ar[d] \\
{\rm Map}_*(M_{l,0},BG) \ar@{=}[r] & {\rm Map}_*(M_{l,0},BG)\ar[d]_{i^*} \ar[r] & \ast\\
&{\rm Map}_*(S^3,BG)
}
\end{equation}
where rows and columns are fibrations. Here we have identified ${\rm{Map}}_*(S^4\vee S^7,BG)$ with $${\rm{Map}}_*(S^4,BG)\times{\rm{Map}}_*(S^7,BG),$$ so that $p_2^*$ is the projection and $i^*$ is the inclusion. Notice that the map $q^*$ induces a bijection between path components as $i^*$ does since $\pi_6(G)=0$.  For every $k\in \mathbb Z$ there is a homotopy equivalence between the path components $\theta_k:\Omega^4_0BG\to\Omega^4_kBG$ defined by $\omega\mapsto\mu\circ(\omega\times k_0)\circ\Delta,$ where $\mu$ is a homotopy multiplication in $\Omega^4BG,$ $\Delta$ is the diagonal map and $k_0$ is a choice of base point in $\Omega_k^4BG.$ 
Thus there is a homotopy commutative diagram
\begin{equation}\label{d:fib2}
\xymatrix{
{\rm Map}_*(\Sigma Y_{l},BG)\ar[r] \ar[d] & \Omega^4BG \ar[r]^-{\gamma^*} \ar[d]_-{ (0+\gamma)^*} &  \Omega^7BG \ar@{=}[d] \\
\Omega^4_0BG \ar[r]^-{i^*} \ar[d]_-{\theta_k} & \Omega^4_0BG \times\Omega^7BG \ar[r]^-{{p_2}^*} \ar[d]_-{\theta_k\times\mathbbm 1} &\Omega^7BG \ar@{=}[d] \\
\Omega^4_kBG \ar[r]^-{i^*}  & \Omega^4_kBG \times\Omega^7BG \ar[r]^-{{p_2}^*}
&\Omega^7BG 
}
\end{equation}
Therefore, by the homotopy commutativity of \eqref{d:fib2}, the restriction of $q^*$ in \eqref{d:fib1} to the $k$-th component
generates a homotopy commutative diagram 
\begin{equation}
\xymatrix{
{\rm Map}_*(\Sigma Y_{l},BG)\ar[r] \ar[d] & \Omega^4BG \ar[r]^-{\gamma^*} \ar[d]_-{(\theta_k\times\mathbbm1) \circ (0+\gamma)^*} &  \Omega^7BG \ar@{=}[d] \\
\Omega^4_kBG \ar[r]^-{i^*} \ar[d]_-{\pi^*} & \Omega^4_kBG \times\Omega^7BG \ar[r]^-{{p_2}^*} \ar[d]_-{q^*} &\Omega^7BG \ar[d] \\
{\rm Map}_*^k(M_{l,m},BG) \ar@{=}[r] & {\rm Map}_*^k(M_{l,m},BG) \ar[r]& \ast
}
\end{equation}
where each row and column is a fibration sequence. This shows that there are homotopy equivalences
$$F^k_{l,0} \simeq {\rm{Map}}_*(\Sigma Y_l,BG)\simeq{\rm{Map}}_*(Y_l,\Omega BG)\simeq{\rm{Map}}_*(Y_l,G)$$
for all $k\in\mathbb Z$.
\end{proof}

If  $Y$ is an $H$-group, or if $X$ is a co-$H$-group, then all the path components of Map$_*(X,Y)$ are homotopy equivalent. So for instance, if  $m=1$, then $M_{l,1}\simeq S^7.$ In this case for any  $k,k'\in[M_{l,1},BG]$, the path components ${\rm{Map}}^{k}_*(M_{l,1},BG)$ and ${\rm{Map}}^{k'}_*(M_{l,1},BG)$ are homotopy equivalent and, as a consequence, so are the pointed gauge groups. When $M_{l,m}$ is not homotopy equivalent to $S^7$, it is not known if the path components of ${\rm{Map}}_*(M_{l,m},BG)$ have the same homotopy type. We prove a result on the homotopy types of the pointed gauge groups over manifolds $M_{l,m}$ with torsion-free homology and $m=0$.

\begin{theorem}\label{t:PGG}
Let $G$ be a simply connected simple compact Lie group with $\pi_6(G)\cong 0$. Let $M_{l,0}$ be the total space of an $S^3$-bundle over $S^4$ with a cross section. Then for all $k\in \mathbb{Z}$
there is a homotopy equivalence $$\mathcal{G}_*^k(M_{l,0})\simeq_{} \Omega^4G\times {\rm{Map}}_*(Y_l,G).$$  
In particular, if $l\equiv 0\pmod{12}$ then there is a homotopy equivalence
$$\mathcal{G}^k_*(M_{l,0})\simeq\Omega^4G\times\Omega^3G\times \Omega^7G.$$
\end{theorem}
 \begin{proof}
Let $M_{l,0}$ be a manifold with a cross section. Let $\mathcal{G}^k_*(M_{l,0})$ be the gauge group of the principal $G$-bundle over $M_{l,m}$ classified by $k\in\mathbb{Z}.$ 
By Lemma \ref{hofibre} there is a fibration sequence
\begin{equation}\label{eq:fibpcom}
{\rm{Map}}_{*}( Y_{l},G)\to {\rm{Map}}_*^k(S^4,BG){\xrightarrow{\pi^*_k}}{\rm{Map}}_*^k(M_{l,m},BG).
\end{equation}
Extend the fibration sequence to the left. Consider the following part of the fibration
\begin{equation}\label{f:ls3inc}
 \xymatrix{
\Omega {\rm{Map}}_*^0(S^4,BG)\simeq\Omega {\rm{Map}}_*^k(S^4,BG)\ar[r]^-{\Omega\pi_k^{*}}&\Omega{\rm{Map}}_*^k(M_{l,m},BG)\ar[r]&{\rm{Map}}_*(Y_{l},G).
}\end{equation}
Let $s:M_{l,m}\to S^4$ be a cross section, that is, a map such that  the diagram 
\begin{equation}\label{d:section}
\xymatrix{
S^4\ar[r]^s\ar@{=}[rd]&M_{l,m}\ar[d]^\pi\\
&S^4
}
\end{equation}
commutes. Applying the functor ${\rm{Map}}^*(-,BG)$ to the diagram \eqref{d:section} we obtain the following homotopy commutative diagram
\begin{equation}\label{d:indsection}
\xymatrix{
{\rm{Map}}_*(S^4,BG)\ar@{=}[rd]&{\rm{Map}}_*(M_{l,m},BG)\ar_{s^{*}}[l]\\
&{\rm{Map}}_*(S^4,BG)\ar[u]_{\pi^*}.
}
\end{equation}
Take $k$-th components to get a similar diagram, so $s^*_k$ is a right homotopy inverse of $\pi^*_k$. Thus after looping there is a homotopy equivalence
$$\Omega {\rm{Map}}^k_*(M_{l,0},BG)\simeq\Omega{\rm{Map}}_*^0(S^4,BG)\times{\rm{Map}}_*(Y_l,G).$$
We can identify $\Omega {\rm{Map}}_*^0(S^4,BG)\simeq\Omega^4 G$ and $\mathcal{G}_{*}^k(M_{l,m})\simeq \Omega{\rm{Map}}_*^k(M_{l,m},BG)$. Putting things together we obtain $$\mathcal{G}^k_*(M_{l,0})\simeq\Omega^4G\times{\rm{Map}}_*(Y_l,G).$$ Finally, by Proposition \ref{p:susp} when $l\equiv 0\pmod{12}$ we have $\Sigma Y_l\simeq S^4\vee S^8$ and therefore we get
\[
\mathcal{G}^k_*(M_{l,0})\simeq\Omega^4G\times\Omega^3G\times \Omega^7G.\qedhere
\]
 \end{proof}

We can use Theorem \ref{t:PGG} to compute homotopy groups of the pointed gauge groups.

\begin{cor}\label{c:pgghg}
For all $k\in\mathbb Z$ and for all $n\geq0$ there are isomorphisms
$$\pi_n(\mathcal G^k_*(M_{l,0}))\cong\pi_{n+4}(G)\oplus\pi_n({\rm{Map}}_*(Y_l,G)).$$
Further, if $l\equiv 0\pmod {12}$ then 
$$\pi_n(\mathcal G^k_*(M_{l,0}))\cong\pi_{n+4}(G)\oplus\pi_{n+3}(G)\oplus\pi_{n+7}(G).$$ 
\qed
\end{cor}
Corollary \ref{c:pgghg} shows that homotopy groups of $\mathcal G^k_*(M_{l,0})$ can be obtained using information of the homotopy groups of the structure group $G$ of the principal bundles.

Now we look at the evaluation map to obtain a homotopy decomposition of the unpointed gauge groups. The restriction of evaluation map to the $k$-th component defines a fibration sequence
\begin{equation} \label{e:starstar}
\Omega{\rm{Map}}^k(M_{l,m},BG)\longrightarrow G\overset{\partial^k}{\longrightarrow}{\rm{Map}}^k_*{(M_{l,m},BG)}\longrightarrow{\rm{Map}}^k(M_{l,m},BG)\overset{ev_k}{\longrightarrow}BG
\end{equation}
where $\partial^k$ is the connecting  map. Thus the gauge group $\mathcal{G}^k(M_{l,m})\simeq\Omega{\rm{Map}}^k(M_{l,m},BG)$ appears as the homotopy fibre of the connecting map $\partial^k$. Hence, it is expected that the properties of $\partial^k$ determine the homotopy type of the gauge groups over the manifolds $M_{l,m}$.

By Proposition \ref{bundles}, if $m=0$ the projection $M_{l,0}\xrightarrow{\pi}S^4$ induces an bijection between path components of Map$(M_{l,0},BG)$ and those of Map$(S^4,BG).$ Therefore, the evaluation map induces a commutative diagram
\begin{equation}\label{evdiag0}
\xymatrix{ 
G\ar[r]^{\phi_k\qquad}\ar@{=}[d]&{\rm{Map}}_{*}^k(S^4,BG)\ar[r]\ar[d]^{\pi_k^{*}}&
{\rm{Map}}^k(S^4,BG)\ar[r]^-{ev_k}\ar[d]^{\pi^{*}_k}&BG\ar@{=}[d]\\
G\ar[r]^-{\partial_k}&{\rm{Map}}_{*}^k(M_{l,0},BG)\ar[r]&{\rm{Map}}^k(M_{l,0},BG)
\ar[r]^-{ ev_k}
&BG
}
\end{equation}
which defines the map $\phi_k$. We write $\Omega {\rm{Map}}^k(S^4,BG)\simeq\mathcal{G}^k(S^4).$

\begin{proof}[Proof of Theorem \ref{t:UGGTF}]
We argue along the lines of \cite{Th1}. Consider the restriction of the map $$\pi^*:{\rm{Map}}_*(S^4,BG)\to{\rm{Map}}_*(M_{l,m},BG)$$ to the $k$-th component. By Lemma \ref{hofibre} there is a fibration sequence 
\begin{equation}\label{eq:fibresev}
\Omega {\rm{Map}}^k_*(M_{l,0},BG)\xrightarrow{\delta^*}{\rm{Map}}_{*}(Y_{l},G)\to {\rm{Map}}_*^k(S^4,BG){\xrightarrow{\pi^*_k}}{\rm{Map}}_*^k(M_{l,0},BG).
\end{equation}
We identify ${\rm{Map}}^k_*(M_{l,0},BG)\simeq\mathcal G_*^k(M_{l,m}).$ The left square in \eqref{evdiag0} along with \eqref{eq:fibresev} induce the following homotopy commutative diagram 
 \begin{equation}\label{gencd}
\xymatrix{
{*}\ar[r]\ar[d]&\mathcal G_*^k(M_{l,0})\ar@{=}[r]\ar[d]&\mathcal G_*^k(M_{l,0})\ar[d]^{\delta^*}\\
\mathcal{G}^k(S^4)\ar[r]^-{}\ar@{=}[d]&\mathcal{G}^k(M_{l,0})\ar[r]^-h\ar[d]&{\rm{Map}}_{*}(Y_{l},G)\ar[d]^{}\\
\mathcal{G}^k(S^4)\ar[r]^{}\ar[d]&G\ar[r]^-{\phi_k}\ar[d]^-{\partial_k}&{\rm{Map}}^k_{*}(S^4,BG)\ar[d]^{\pi^*_k}\\
{*}\ar[r]&{\rm{Map}}^{k}_{*}(M_{l,0},BG)\ar@{=}[r]&{\rm{Map}}^{k}_{*}(M_{l,0},BG)
}
\end{equation}
which defines the map $h$.

By Theorem \ref{t:PGG} the map $\delta^*$ has a right homotopy inverse which implies that the map $h$ also has a right homotopy inverse.
The group structure on $\mathcal{G}^k(M_{l,0})$ allows to multiply to obtain a composite
\begin{equation*}
\mathcal{G}^k({S}^4)\times{\rm{Map}}_*(\Sigma Y_l,BG)\rightarrow \mathcal{G}^k(M_{l,0})\times \mathcal{G}_k(M_{l,0})\rightarrow \mathcal{G}^k(M_{l,0}),
\end{equation*}
which is a homotopy equivalence.

If $l\equiv 0 \pmod{12}$ then by Lemma \ref{hofibre} $${\rm{Map}}_*(\Sigma Y_l,BG )\simeq \Omega^3G\times \Omega^7G.$$
as required.
\end{proof}

 \begin{remark}
 The evaluation fibration induces exact sequences of homotopy groups
\begin{equation}\label{e:hgev}
\cdots\rightarrow\pi_n(\mathcal G^k_*(M_{l,0}))\xrightarrow{i^*}\pi_n(\mathcal G^k(M_{l,0}))\xrightarrow{ev^*}\pi_n(G)\rightarrow\cdots
\end{equation}
for all $k\in\mathbb Z$. Given a simply connected simple compact Lie group we have $\pi_n(G)=0$ for $n\leq 2$, which implies that the map $i^{*}$ is an isomorphism for $n\leq2$. We can use these isomorphisms and Corollary  \ref{c:pgghg} to compute the path components of unpointed gauge groups $\mathcal G^k(M_{l,m})$. For example, for the manifolds $M_{l,0}$ such that $l\equiv0\pmod {12}$, we compute the path components of the gauge groups using information of the homotopy groups of Lie groups as given in \cite{Jam2}:
\begin{equation*}
\pi_0(\mathcal G^k(M_{l,0}))=\begin{cases}
\mathbb Z\times\mathbb Z\times\mathbb Z & G=Spin(8)\\
\mathbb Z\times\mathbb Z\times\mathbb Z_2 & G=Sp(n)(n\geq2),Spin(5)\\
\mathbb Z\times\mathbb Z & G=SU(n)(n\geq4), Spin(m)(m=6,7 \text{ or } m \geq 9)\\
\mathbb Z & G=F_4,E_6,E_7,E_8.
\end{cases}
\end{equation*}
\end{remark}

\subsection{Non-torsion-free case}
\label{s:torsion}

Now we focus on the case of gauge groups of principal $G$-bundles over manifolds $M_{l,m}$ for $m\geq 2$, which have non-torsion-free homology. To obtain homotopy decompositions it will be required that spaces are localised at a prime $p\geq 5.$ In this case we will obtain results for the loop spaces of the gauge groups, $\Omega \mathcal{G}^{ k}(M_{l,m}).$ As in the torsion-free case, the following lemma will be required to identify the spaces that appear in the homotopy decomposition of the gauge groups.

\begin{lemma}\label{l:diagram2}
There is a homotopy commutative diagram of cofibrations
\begin{equation}\label{d:deltadiag2}
\xymatrix{
{*}\ar[r]\ar[d]&S^7\ar@{=}[r]\ar@{^{(}->}[d]&S^7\ar[d]^{\gamma}\ar[r]&{*}\ar[d]\ar[r]&S^8\ar@{=}[r]\ar[d]&S^8\ar[d]^{\Sigma\gamma}\\
M_{l,m}\ar@{=}[d]\ar[r]^-{g}&S^4\vee S^7\ar[r]^-{\delta}\ar[d]^{pinch}&S^4\ar[r]^{}\ar[d]^{}&\Sigma M_{l,m}\ar@{=}[d]\ar[r]&S^5\vee S^8\ar[r]^-{\Sigma \delta}\ar[d]&S^5\ar[d]\\
M_{l,m}\ar[r]^{\pi}&S^4\ar[r]^-{q}&C_{l,m}\ar[r]^-{}&\Sigma M_{l,m}\ar[r]&S^5\ar[r]^-{\Sigma q}&\Sigma C_{l,m}
}
\end{equation}
which defines the space $C_{l,0}$ and the maps $\gamma$ and $\delta$. Furthermore, the map $S^5\xrightarrow{\Sigma q}\Sigma C_{l,m}$ is identified with the composite $$S^5\xrightarrow{m} S^5\hookrightarrow \Sigma C_{l,m}$$ and after localisation at $p\geq 5$ there are homotopy equivalences
$$\Sigma C_{l,m} \simeq_{(p)}{S^5\vee S^9}.$$
\end{lemma}
\begin{proof}
 Arguing along the lines of Lemma \ref{l:diagram} there is a homotopy commutative diagram
 \begin{equation}\label{d:pinch2}
\xymatrix{
M_{l,m}\ar[r]^-{}\ar[d]_-{\pi}&S^4\vee S^7\ar[ld]^-{pinch}\\
S^4&
}
\end{equation}
that we can extend to obtain a homotopy commutative diagram as shown in \eqref{d:deltadiag2}. Thus $\delta=\beta+\gamma$ where $\beta\in\pi_4(S^4)$ and $\gamma\in\pi_7(S^4)\cong\mathbb{Z}\times\mathbb{Z}_{12}$. Using the long exact sequence that the middle row of \eqref{d:deltadiag2} induces in homology, we can see that $\beta$ is the degree $m$ map. Thus we can identify the map $\Sigma q$ with the composite
$$S^5\xrightarrow{m} S^5\hookrightarrow \Sigma C_{l,m}.$$ The homotopy group $\pi_8(S^5)$ becomes trivial after localisation at a prime $p\geq 5$ \cite{Tod}. Since $\Sigma C_{l,m}$ is the homotopy cofibre of the map $\Sigma\gamma\in\pi_8(S^5)$, after localisation at a prime $p\geq5$ there is a homotopy equivalence
\[
\Sigma C_{l,m}\simeq_{(p)} S^5\vee S^9. \qedhere
\]
\end{proof}

By Proposition \ref{bundles}, the projection map induces a surjection between path components 
\begin{equation}
\pi^*:[S^4,BG]=\pi_0({\rm{Map}}_*(S^4,BG)) \rightarrow [M_{l,m},BG]=\pi_0({\rm{Map}}_*(M_{l,m},BG)).
\end{equation}
Moreover, the projection map induces the following fibration sequences
$${\rm{Map}}_*(\Sigma C_{l,m},BG)\to {\rm{Map}}_*(S^4,BG){\xrightarrow{\pi^*}}{\rm{Map}}_*(M_{l,m},BG),$$ $$F^k_{l,m}\to {\rm{Map}}_*^{\bar k}(S^4,BG){\xrightarrow{\pi^{*}_k}}{\rm{Map}}_*^k(M_{l,m},BG),$$ where $\pi_k^*$ is the restriction of $\pi^*$ to the $k$-th component and $F^k_{l,m}$ is the corresponding homotopy fibre. 
More precisely, if $m\geq2$ then ${\rm{Map}}^{}(M_{l,m},BG)$ has $m$ components and $\pi^*$ sends the $\bar k$-th component of ${\rm{Map}}(S^4,BG)$ to the $k$-th component of ${\rm{Map}}(M_{l,m},BG),$ where $k$ is the reduction mod  $m$ of $\bar k.$

\begin{lemma}\label{hofibret}
After localisation at a prime $p\geq 5$, there is a fibration sequence
$$\Omega^9BG\times\Omega^5BG\xrightarrow {* \times m^*}\Omega^5BG \xrightarrow{\underline{\pi}^*_k} \Omega{\rm{Map}}_k^*(M_{l,m},BG),$$ 
where $m_*$ is the $m$-th power map, and the map $\underline{\pi}_*^k$ is identified with the composite
$$\Omega{\rm Map}_*^0(S^4,BG) \xrightarrow{\Omega\theta_k} \Omega{\rm Map}_*^k(S^4,BG) \xrightarrow{\Omega\pi^*_k} \Omega{\rm{Map}}^k_*(M_{l,m},BG),$$
where $\Omega\theta_k$ is a homotopy equivalence.
\end{lemma}
\begin{proof}
Applying the functor Map$_*(-,BG)$ to the diagram in Lemma \ref{d:deltadiag2}, we obtain a homotopy commutative diagram of fibrations
\begin{equation}
\xymatrix{
{\rm Map}_*(C_{l,m},BG)\ar[r] \ar[d]_-{q^*} & {\rm Map}_*(S^4,BG) \ar[r]^-{\gamma^*} \ar[d]_-{(m+\gamma)^*} & {\rm Map}_*(S^7,BG) \ar@{=}[d] \\
{\rm Map}_*(S^4,BG) \ar[r]^-{} \ar[d]_-{\pi^*} & {\rm Map}_*(S^4 \vee S^7,BG) \ar[r]^-{{p_2}^*} \ar[d]_-{g^*} & {\rm Map}_*(S^7,BG) \ar[d] \\
{\rm Map}_*(M_{l,m},BG) \ar@{=}[r] & {\rm Map}_*(M_{l,m},BG)\ar[d]_-{i^*} \ar[r] & \ast\\
&{\rm Map}_*(S^3,BG).
}
\end{equation}
We can identify ${\rm{Map}}_*(S^4\vee S^7,BG)$ with ${\rm{Map}}_*(S^4,BG)\times{\rm{Map}}_*(S^7,BG)$. The following diagram is obtained by restricting the map $i^*$ to the $k$-th path component and composing with the map $$\theta_k\times \mathbbm1:\Omega^4_0BG\times\Omega^7BG\to\Omega^4_kBG\times\Omega^7BG,$$ where $\theta_k$ is the homotopy equivalence given by $\omega\mapsto\mu\circ(\omega\times k_0)\circ\Delta$ for fixed $k_0 \in \Omega^4_kBG$, and $\mathbbm 1$ is the identity map on $\Omega^7BG$:
\begin{equation}
\xymatrix{
{\rm{Map}}_*(C_{l,m},BG)\ar[r] \ar[d] & \Omega^4BG \ar[r]^-{\gamma^*} \ar[d]_-{(\theta_k\times\mathbbm1) \circ( m+\gamma)^*} &  \Omega^7BG \ar@{=}[d] \\
 \bigsqcup_{i\in\mathbb{Z}}\Omega^4_{im+k}BG \ar[r]^-{} \ar[d]_-{\pi^*_k} & (\bigsqcup_{i\in\mathbb{Z}}\Omega^4_{im+k}BG )\times\Omega^7BG \ar[r]^-{{p_2}^*} \ar[d]_-{g^*_k} &\Omega^7BG \ar[d] \\
{\rm Map}_*^k(M_{l,m},BG) \ar@{=}[r] & {\rm Map}_*^k(M_{l,m},BG) \ar[r]& \ast
}
\end{equation}
Here all rows and the middle and right columns, and hence the left column, are fibrations. Note that since the projection map $g^*:{\rm{Map}}_*(S^4,BG)\to{\rm{Map}}_*(M_{l,m},BG)$ induces a surjection in path components, the homotopy fibre of $i^*$ restricted to the $k$-th component is not path connected. Choose a basepoint $\tilde k:M_{l,0}\to BG$ in $\bigsqcup_{i\in\mathbb{Z}}\Omega^4_{im+k}BG$. Then after looping we get  $$\Omega(\bigsqcup_{i\in\mathbb{Z}}\Omega^4_{im+k}BG)=\Omega(\Omega_{\tilde k}^4BG),$$
and also
$$\Omega(\bigsqcup_{i\in\mathbb{Z}}\Omega^4_{im+k}BG)\times\Omega^7BG)=\Omega(\Omega_{\tilde k}^4BG\times\Omega^7BG),$$
where $\Omega_{ \tilde k}^4BG$ is the path component containing the map $\tilde k$. Observe that this result holds integrally. Applying the functor $\Omega(-)$ to the previous diagram, we  obtain the following homotopy commutative diagram

\begin{equation}\label{d:loopgauge}
\xymatrix{
\Omega {\rm{Map}}_*^0(C_{l,m},BG)\ar[r] \ar[d]_-{\tilde q_k^*} & \Omega^5BG \ar[r]^-{\Omega\gamma^*} \ar[d]_-{\Omega(m+\gamma)^*} &  \Omega^8BG \ar@{=}[d] \\
\Omega^5BG \ar[r] \ar[d]_-{\underline{\pi}^*_k} & \Omega^5BG \times\Omega^8BG \ar[r]^-{p_2^*} \ar[d]_-{\Omega g_k^*} &\Omega^8BG \ar[d] \\
\Omega{\rm Map}_*^k(M_{l,m},BG) \ar@{=}[r] & \Omega{\rm Map}_*^k(M_{l,m},BG) \ar[r]& \ast
}
\end{equation}
where we have identified already $\Omega (\Omega^7BG)$ with $\Omega^8BG$, and $\Omega(\Omega^4_k BG)$ with $\Omega(\Omega^4_0BG)\simeq \Omega(\Omega^4BG)\simeq \Omega^5BG$ for all $k$. Take adjoints in the diagram \eqref{d:loopgauge}. 
The adjoint of $\Omega \gamma^*$ is homotopic to $(\Sigma\gamma)^*$.
 Localising at a prime $p\geq5$ and using Lemma \ref{l:diagram2} we obtain a string of homotopy equivalences
$$\Omega {\rm{Map}}_*^0(C_{l,m},BG)\simeq_{}{\rm{Map}}_*(\Sigma C_{l,m},BG)\simeq_{(p)}\Omega^9BG\times\Omega^5BG,$$
and $\tilde q_k^*\simeq *\times m^*,$ where $m^*$ is the $m$-th power map.
\end{proof}

Now we give results on the homotopy decomposition of the pointed gauge groups. Recall that for any space $X$, the cofibration sequence $S^n\xrightarrow kS^n\to P^{n+1}(k)$ induces the following fibration sequence
$${\rm{Map}}_*(P^{n+1}(k),X)\to \Omega^nX\xrightarrow{k^*}\Omega^nX,$$
where $k^*$ is the $k$-th power map. Let $\Omega^nG\{k\}:={\rm{Map}}_*(P^{n+1}(k),BG).$ The following is a result on  the pointed gauge groups.

\begin{theorem}\label{t:PGG2}
Let $G$ be a simply connected simple compact Lie group with $\pi_6(G)= 0$. Let $ M_{l,m}$ be the total space of an $S^3$-bundle over $S^4$ such that $m\geq2.$ Localising at a prime $p\geq5$ there are homotopy equivalences $$\mathcal{G}_*^0(M_{l,m})\simeq_{(p)}\Omega^3G\{p^r\}\times \Omega^7G,$$ and for all $k\in\mathbb Z_m$ $$\Omega\mathcal{G}_*^k(M_{l,m})\simeq_{(p)}\Omega^4G\{p^r\}\times \Omega^8G,$$ where $r = v_p(m)$ is the valuation of $m$ at $p$.
\end{theorem}

\begin{proof}
Observe that since Map$_*(M_{l,m},BG)$ is a pointed space with base point the constant map $*:M_{l,m}\to BG$, there is a homotopy equivalence $\Omega {\rm{Map}}_*^0(M_{l,m},BG)\simeq \Omega ({\rm{Map}}_*(M_{l,m},BG))$. Localise at a prime $p\geq 5$ and let $v_p(m)=r$. By Proposition \ref{torman}, there is a homotopy equivalence $$\Sigma M_{l,m}\simeq_{(p)} P^5(p^r)\vee S^8.$$ Thus we obtain a string of homotopy equivalences
 $$\Omega {\rm{Map}}_*^0(M_{l,m},BG)\simeq_{(p)} \Omega ({\rm{Map}}_*(M_{l,m},BG))\simeq_{(p)} {\rm{Map}}_*(\Sigma M_{l,m},BG)\simeq_{(p)}{\rm{Map}}_*(P^5(p^r)\vee S^8,BG).$$
Taking adjoints we obtain $${\rm{Map}}_*(P^5(p^r)\vee S^8,BG)\simeq_{(p)} {\rm{Map}}_*(P^4(p^r)\vee S^7,G)\simeq_{(p)}{\rm{Map}}_*(P^4(p^r),G)\times {\rm{Map}}_*( S^7,G).$$
Since $\mathcal{G}_*^0(M_{l,m})\simeq_{(p)}\Omega {\rm{Map}}_*^0(M_{l,m},BG)$, we get $\mathcal{G}_*^0(M_{l,m})\simeq_{(p)} \Omega^3G\{p^r\}\times \Omega^7G.$

Now let $\mathcal G^k_*(M_{l,m})$ be a gauge group with $k\neq0$. By Lemma \ref{hofibret} there is a fibration sequence
$$\Omega^9BG\times\Omega^5BG\xrightarrow {* \times m^*}\Omega^5BG \xrightarrow{\underline{\pi}^*_k} \Omega{\rm{Map}}^k_*(M_{l,m},BG),$$ 
where $m^*$ is the $m$-th power map, and the map $\underline{\pi}^*_k$ is identified with the composite
$$\Omega{\rm Map}_*^0(S^4,BG) \xrightarrow{\Omega\theta_k} \Omega{\rm Map}_*^k(S^4,BG) \xrightarrow{\Omega\pi^*_k} \Omega{\rm{Map}}^k_*(M_{l,m},BG),$$
where $\theta_k: \Omega^4_0BG \to \Omega^4_kBG$ is a homotopy equivalence.
Note that the homotopy fibre of the map $* \times m^*$ is homotopy equivalent to $\Omega^2{\rm{Map}}_*^k(M_{l,m},BG)$, which can be identified with the loop space of the pointed gauge group, $\Omega\mathcal G^k_*(M_{l,m}).$ Now identifying $\Omega^5_0BG$ with $*\times\Omega^5_0BG$ it is straightforward to check that there is a homotopy equivalence $$\Omega\mathcal G^k_*(M_{l,m})\simeq_{(p)}\Omega^4G\{p^r\}\times \Omega^8G,$$
as required.  
\end{proof}
Given a nilpotent space $X$, let $(\pi_m(X))_{(p)}$ be the localisation of the homotopy group $\pi_m(X)$ at a prime $p$. Using the theory of homotopy groups with coefficient (see \cite[Chapter 1]{Nei}) and Theorem \ref{t:PGG2} we can compute the homotopy groups $\pi_n((\mathcal G_*^0(M_{l,m}))_{(p)})$ and $\pi_n((\Omega\mathcal G_*^k(M_{l,m}))_{(p)})$. For $n+j\geq2$, let $\pi_{n+j}(G;\mathbb Z_{p^r})=[P^{n+j}(p^r),G]$.
 \begin{cor}\label{c:hg_ggt}
 Let $p\geq5$ be a prime and $v_p(m)=r$. For every $n\geq 0$  there are isomorphisms
\begin{enumerate}[(1)]
\item $\pi_n((\mathcal G_*^k(M_{l,m}))_{(p)})\cong\pi_{n+3}(G;\mathbb Z_{p^r})\oplus(\pi_{n+7}(G))_{(p)}$, for $k=0$;
\item $\pi_n((\Omega\mathcal G_*^k(M_{l,m}))_{(p)})\cong\pi_{n+4}(G;\mathbb Z_{p^r})\oplus(\pi_{n+8}(G))_{(p)}$, for all $k\in\mathbb Z_m$. \qed
 \end{enumerate}
\end{cor}

Now by Proposition \ref{bundles}, $M_{l,m}\xrightarrow{\pi}S^4$ induces a surjection $\pi^*:[S^4,BG]\rightarrow[M_{l,m},BG]$  if $m>1.$ Therefore, by the naturality of the evaluation map, we obtain a commutative diagram
\begin{equation}\label{evdiag}
\xymatrix{ 
\Omega G\ar[r]^-{\Omega\phi^k}\ar@{=}[d]&\Omega{\rm{Map}}_{*}^k(S^4,BG)\ar[r]\ar[d]^-{\pi^{*}_k}&
\Omega{\rm{Map}}^k(S^4,BG)\ar[r]^-{ev_k}\ar[d]^-{\pi^{*}_k}&G\ar@{=}[d]\\
\Omega G\ar[r]^-{\Omega\partial^k}&\Omega{\rm{Map}}_{*}^k(M_{l,m},BG)\ar[r]&\Omega{\rm{Map}}^k(M_{l,m},BG)
\ar[r]^-{ ev_k}
&G
}
\end{equation}
which defines the map $\Omega\phi^k$. Recall that if $m\geq2$ then ${\rm{Map}}^{}(M_{l,m},BG)$ has $m$ components and $${\rm{Map}}(S^4,BG)\xrightarrow{\pi^*_k}{\rm{Map}}(M_{l,m},BG)$$ sends the $\bar k$-th component of ${\rm{Map}}(S^4,BG)$ to the $\bar k$-th component of ${\rm{Map}}(M_{l,m},BG),$ where $ k$ is the mod  $m$ reduction of $\bar k.$

\begin{proof}[Proof of Theorem \ref{t:UGGT}]
 Localise all spaces at a prime $p\geq5$, so that  $\pi_6(G) \cong 0$ for any simply connected simple compact Lie group. Suppose that $v_p(m)=0.$ Then $M^4_{l,m}=P^4(m)\simeq*$ and therefore $M_{l,m}\simeq S^7.$ Thus in this case there is only one principal $G$-bundle over $M_{l,m}$ up to isomorphism, namely, the trivial bundle. Since the map $ev_0$ in \eqref{e:starstar} has a section and this is a principal fibration we obtain a $p$-local homotopy equivalence
$$\mathcal{G}^0(M_{l,m})\simeq_{(p)} \Omega^7G\times G.$$
Now suppose that $v_p(m)\geq1.$ By Theorem \ref{t:PGG2}, there is a homotopy equivalence $$\Omega\mathcal G^k_*(M_{l,m})\simeq_{(p)}\Omega^8G\times\Omega^4G\{m\}.$$ Moreover, by Lemma \ref{hofibret} there is a fibration sequence
\begin{equation}\label{loop}
\xymatrix{
\Omega^8G\times\Omega^4G\{m\}\ar[r]^-{\delta^*}&\Omega^8G\times \Omega^4G\ar[r]^-{*\times m^*}&\Omega^4G\ar[r]^-{\underline{\pi}_*^k}&\Omega{\rm{Map}}^k_*(M_{l,m},BG)\simeq\mathcal G^k_*(M_{l,m}).
}
\end{equation}
Here we identify $\Omega^5BG$ with $\Omega^4G$ and $\Omega^9BG\times\Omega^5 BG$ with $\Omega^8G\times\Omega^4G$. This implies that $\delta^*\simeq\mathbbm{1}\times j$, where $j$ is the inclusion map. The left-hand side square of the homotopy commutative diagram \eqref{evdiag} shows that the bottom square of the diagram \eqref{diagloop} homotopy commutes. The whole diagram \eqref{diagloop} is generated taking fibres along the maps $\Omega\phi^k$, $\Omega\partial^k$ and $\underline{\pi}^*_k$, which defines the map $h$.

\begin{equation}\label{diagloop}
\begin{gathered}
\xymatrix{
{*}\ar[r]\ar[d]&\Omega^8G\times\Omega^4G\{m\}\ar@{=}[r]\ar[d]^-{j'}&\Omega^8G\times\Omega^4G\{m\}\ar[d]^{{\mathbbm{1}}\times j}\\
\Omega\mathcal{G}^k(S^4)\ar[r]^-{}\ar@{=}[d]^{}&\Omega\mathcal{G}^k(M_{l,m})\ar[r]^-h\ar[d]^{}&\Omega^8G\times\Omega^4G\ar[d]^{*\times m^{*}}\\
\Omega\mathcal{G}^k(S^4)\ar[r]^{}\ar[d]&\Omega G\ar[r]^-{\Omega\phi^k}\ar[d]^-{\Omega\partial^k}&\Omega^4G\ar[d]^{\underline{\pi}^*_k}\\
{*}\ar[r]&\Omega{\rm{Map}}^{k}_{*}(M_{l,m},BG)\ar@{=}[r]&\Omega{\rm{Map}}^{k}_{*}(M_{l,m},BG)
}
\end{gathered}
\end{equation}

Let $\bar h$ be the composite
$$\xymatrix{
\Omega \mathcal{G}^k(M_{l,m})\ar[r]^-h&\Omega^8G\times \Omega^4G\ar[r]^-{p_1}&\Omega^8G,
}$$
where $p_1$ is the projection onto the first factor. The top square of \eqref{diagloop} shows that $\bar h$ has a right homotopy inverse. Let $X_k$ be the homotopy fibre of the map $\bar h.$ Then there is a homotopy equivalence 
\begin{equation}\label{equivtor}
 \Omega\mathcal{G}_k(M_{l,m})\simeq_{(p)} X_k\times \Omega^9BG.
 \end{equation}
Finally from \eqref{diagloop} and \eqref{equivtor} there exists a homotopy pullback square
\begin{equation*}\label{diagram}
\begin{gathered}
\xymatrix{
&\Omega^4G\{m\}\ar@{=}[r]\ar[d]&\Omega^4G\{m\}\ar[d]^{}\\
\Omega\mathcal{G}^k(S^4)\ar[r]^-{}\ar@{=}[d]^{}&X_k\ar[r]\ar[d]^{}&\Omega^4G\ar[d]^{m^{*}}\\
\Omega\mathcal{G}^k(S^4)\ar[r]^{}&\Omega G\ar[r]^-{\Omega\phi^k}&\Omega^4G.
}
\end{gathered}
\end{equation*}

Let $r=v_p(m)$ and let $q^*$ be the connecting map of the fibration sequence
$$\Omega^3G\{m\}\to\Omega^3G\xrightarrow{m^*}\Omega^3G.$$ 
 Observe that after localisation $q^*\circ\Omega \phi^k\simeq q^*\circ\Omega\phi^{k'}$ if $k\equiv k'\pmod{p^r}$.  It follows that if $p^r$ divides $k$ then $q^*\circ\Omega\phi^k$ is nullhomotopic, 
and the map $\Omega\phi^k$ lifts through $m^*$. Therefore, by the properties of the pullback there is a map $\zeta:\Omega G\to X_k$ which is a homotopy section. Thus in this case we have a splitting $X_k\simeq_{(p)} \Omega G\times \Omega^4G\{m\}$.
\end{proof}

\begin{remark}
As in the torsion-free case, the inclusion of the pointed gauge group $\mathcal G_*^k(M_{l,m})\hookrightarrow \mathcal G^k(M_{l,m})$ induces isomorphisms $\pi_n(\mathcal G^k_*(M_{l,0}))\xrightarrow{i^*}\pi_n(\mathcal G^k(M_{l,0}))$ if $n\leq 2$. Using Corollary  \ref{c:hg_ggt} and the exact sequence (see \cite[Chapter 1]{Nei}) 
\begin{equation}
0\to\pi_{n+3}(G)\otimes\mathbb Z_{p^r}\to\pi_{n+3}(G;\mathbb Z)\to Tor(\pi_{n+2}(G),\mathbb Z_{p^r})\to 0
\end{equation}
we compute the path components of the $p$-localisations of $\mathcal G^0(M_{l,m})$ at $p\geq5$ 
\begin{equation*}
\pi_0((\mathcal G^0(M_{l,m}))_{(p)})=\begin{cases}\mathbb Z_{p^r}\times\mathbb Z_{(p)}\times\mathbb Z_{(p)}& G=Spin(8)\\
\mathbb Z_{p^r}\times\mathbb Z_{(p)}& G=Sp(n)(n\geq2),SU(n)(n\geq4), \\ & Spin(m)(m=5,6,7\text{ or }m\geq9)\\
\mathbb Z_{p^r} & G=F_4,E_6,E_7,E_8.
\end{cases}
\end{equation*}
Notice that if $k\neq0$, we cannot compute $\pi_0((\mathcal G^k(M_{l,m}))_{(p)})$ with our results.
\end{remark}

\section{Counting homotopy types of gauge groups over $S^7$}
\label{s:homos7}

In this section we discuss the classification of the homotopy types of the gauge groups over manifolds $M_{l,m}$ for $m=1$. As all manifolds $M_{l,1}$ are homotopy equivalent to $S^7,$ the following results will be expressed in terms of $S^7.$ The set $Prin_G(S^7)$ of isomorphism classes of principal $G$-bundles over $S^7$ is in one-to-one correspondence with the set $\langle S^7,BG\rangle$. Observe that, by connectivity of $G$, $\langle S^7,BG\rangle=\pi_6(G).$ In Table \ref {tab1} we collect  information on the homotopy groups $\pi_6(G).$ Here $G^*$ is any of the simply connected simple compact Lie groups not isomorphic to $SU(3)$, $G_2$ or $SU(2) \cong Sp(1).$
\begin{table}[ht]
\begin{center}
\caption{}
\label{tab1}
\begin{tabular}{|c |c c c c|}
\hline
$G$&$SU(2)$&$SU(3)$&$G_2$&$G^*$\\
\hline
$\pi_6(G)$&$\mathbb{Z}_{12}$&$\mathbb{Z}_6$&$\mathbb{Z}_3$&0\\
\hline
\end{tabular}
\end{center}
\end{table}

 Let $P_k\to S^7$ be a principal $G$-bundle classified by the map $k\epsilon$, where $\epsilon$ is a generator of the group $\pi_6(G)$ and $k\in\mathbb Z_{|\pi_6(G)|}.$
 We have seen already that as $S^7$ is a co-$H$-space, there are homotopy equivalences $${\rm{Map}}^k_*(S^7,BG)\simeq{\rm{ Map}}^0_*(S^7,BG),$$  which implies that there exist homotopy equivalences 
 $$\mathcal{G}^k_*(S^7)\simeq\mathcal{G}^0_*(S^7),$$ for all $k\in\pi_6(G)$. In what follows we discuss the results on the homotopy classification of the unpointed gauge groups over $S^7.$
 
Consider the fibration sequence
\begin{equation}\label{ev}
\xymatrix{
\mathcal{G}^k(S^7)\ar[r]&G\ar[r]^-{\bar\partial^k}&{\rm{Map}}^k_*(S^7,BG)\ar[r]&{\rm{Map}}^k(S^7,BG)\ar[r]^-{ev}&BG
}
 \end{equation}
where $ev$ is the evaluation map. 
In \cite{Lng} Lang showed that the adjoint of the connecting map $\bar\partial^k$ of the evaluation fibration is a Whitehead product. Since there is a homotopy equivalence 
$${\rm{Map}}_*(S^7, BG)\simeq {\rm{Map}}_{*}(S^6,G),$$
we restate the result of Lang in terms of Samelson products. 
 \begin{lemma}[Lang \cite{Lng}]\label{l:lang}~
Let $G$ be a simply connected simple compact Lie Group. The adjoint $S^6\wedge G\xrightarrow{\partial_k} G$ of the composite
 $$\partial^k: G\xrightarrow{\bar\partial^{k}}{\rm{Map}}^k_*(S^6,G) \xrightarrow\simeq {\rm{Map}}^0_*(S^6,G)$$
 is homotopic to the Samelson product $\langle k\epsilon,{1}_G\rangle.$
\qed
\end{lemma}

It is clear that the order of $\partial_k$ is bounded by both the number of principal $G$-bundles and the order of $[\Sigma^6G,G]$. There is a general result proved by Theriault \cite[Lemma 3.1]{Th2}, that can be used to get information on the $p$-local homotopy types of the gauge groups. Let $Y$ be an $H$-space with a homotopy inverse, let $k:Y\rightarrow Y$ be the $k$-th power map, and let $F_k$ be the homotopy fibre of the map $k\circ f$, where $f:X\to Y$ is a map of finite order $m$.
\begin{lemma}[Theriault \cite{Th2}] \label{Th2L}
Let $X$ be a space and $Y$ be an $H$-space with a homotopy inverse. Suppose there is a map $X\overset{f}{\rightarrow} Y$ of finite order $m$. If $(m,k)=(m,k')$ then $F_k$ and $F_{k'}$ are homotopy equivalent when localised rationally or at any prime.\qed
\end{lemma}
 Now we are ready to give the proof on the classification of gauge groups on manifolds homotopy equivalent to $S^7$.

\begin{proof}[Proof of Theorem \ref{t:S}] ~
Let $\epsilon$ be a generator of $\pi_6(G)$. Given a principal $G$-bundle over $S^7$ classified by an element $k\epsilon\in\pi_6(G)$, we identify the gauge group $\mathcal G^k(S^7)$ and its classifying space $B\mathcal G^k(S^7)$ with $\Omega{\rm{Map}}^k(S^7,BG)$ and ${\rm{Map}}^k(S^7,BG)$, respectively.
\begin{enumerate}[(1)]
\item Let $G=SU(2)$. We identify the Lie group $SU(2)$ with the unit quaternions $S^3$.  By Lemma \ref{l:lang}, the adjoint $\partial_k$ of the connecting map  $\partial^k$ is homotopic to the Samelson product $\langle k \epsilon,\iota_3\rangle$,   where $\iota_3:S^3\rightarrow S^3$ is the identity map on $S^3$. Observe that $$\langle k \epsilon,\iota_3\rangle\in[\Sigma^6G,G]=[\Sigma^6 S^3,S^3]\cong\pi_9(S^3).$$ According to \cite{Tod}, $\pi_9(S^3)\cong\mathbb{Z}_3$, and so the order of $\partial^k$ is at most 3. From the evaluation fibration
$$\xymatrix{
\Omega^6_0S^3\simeq{\rm{Map}}^k_*(S^7,BS^3)\ar[r]^-{i^{}}&{\rm{Map}}^k(S^7,BS^3)\ar[r]^-{ev}&BS^3
} $$
we obtain the following commutative diagram 
\begin{equation}\label{eseq}
\xymatrix@=0.3in{
\pi_3(S^3)\ar[r]^-{}\ar[dr]_-{\partial_k}&\pi_3(\Omega^6_0 S^3)\ar[r]\ar[d]^{\cong}& \pi_3({\rm{Map}}^k(S^7,BS^3))\ar[r]&0,\\
&[S^6\wedge S^3,S^3]\ar[ur]&&
}
\end{equation}
where $\partial_k(f)=\langle k \epsilon,f\rangle$ for any $f\in\pi_3(S^3)\cong\mathbb{Z}$. Thus $\pi_3({\rm{Map}}^k(S^7,BS^3))$ is isomorphic to the cokernel of $\partial_k$. Linearity of the Samelson product implies that $\langle k \epsilon,\iota_3\rangle\simeq k\langle \epsilon,\iota_3\rangle$. Thus we only have to determine the order of $\langle \epsilon,\iota_3\rangle$, that is, the order of the adjoint of $\partial_1$.

Notice that if $\mathcal{G}^k(S^7)\simeq \mathcal{G}^{k'}(S^7)$ then $\pi_n(\mathcal{G}^k(S^7))\cong\pi_n(\mathcal{G}^{k'}(S^7))$ for all $n\geq0$. In particular, from \eqref{eseq} we obtain $$\pi_2(\mathcal{G}^k(S^7))\cong\pi_3(\mathcal G^k(S^7))\cong \pi_3(\mathcal G^{k'}(S^7))\cong \pi_2(\mathcal{G}^{k'}(S^7)).$$

The Samelson product $\langle \iota_3,\iota_3\rangle\in\pi_6(S^3)$ is a generator of $\pi_6(S^3)\cong\mathbb{Z}_{12}$ \cite{S}. Hence, the adjoint of the map $\partial_1$ is homotopic to the iterated commutator map $\langle\langle\iota_3,\iota_3\rangle,\iota_3\rangle\in\pi_9(S^3)\cong\mathbb Z_3$.
According to \cite[Theorem 2]{KK2}, $SU(2)\cong S^3$ localised at $p = 3$ is nilpotent of class $3$. This implies that, integrally, $\partial_1\simeq\langle\langle\iota_3,\iota_3\rangle,\iota_3\rangle$ is essential and it is a generator of $\pi_9(S^3)$. 
 Let $B=BS^3=\mathbb HP^\infty$ and $\gamma=\langle\langle\iota_3,\iota_3\rangle,\iota_3\rangle$. 
Since the map $\gamma$ is a generator of $\pi_3(\Omega^7B)\cong\pi_9(S^3)\cong\mathbb Z_3$, there are homotopy commutative diagrams
\begin{equation}
\xymatrix@=0.45in{
\mathcal G^k(S^7)\ar[r]\ar[d]^-{\simeq}&S^3\ar[r]^-{k\gamma}\ar@{=}[d]&\Omega^7B\ar[d]^{h}\\
\mathcal G^{k'}(S^7)\ar[r]&S^3\ar[r]^-{k'\gamma}&\Omega^7 B}
\end{equation}
 if and only if $(3,k) = (3,k')$, where $h:\Omega^7B\to\Omega^7B$ is either the identity or the homotopy equivalence defined by the rule $x\mapsto x^{-1}$.

\item Let $G=G_2$ and let $\iota:S^3\hookrightarrow G_2$ be a generator of $\pi_3(G_2)$. The map $\langle \iota,\iota\rangle$ represents a generator of $\pi_6(G_2)$ \cite{Mim2}. Thus we have $\partial^{1}\simeq \langle \langle \iota,\iota\rangle,{1}_{G_2}\rangle:S^6\wedge G_2\overset{}{\longrightarrow} G_2$. Consider the following composite
 \begin{equation*}
\xymatrix@=0.60in{
 \theta:S^6\wedge S^3\ar[r]^{\mathbbm{1}\wedge \iota}&S^6\wedge G_2\ar[r]^-{\langle\langle \iota,\iota\rangle,{1}_{G_2}\rangle}&G_2.
 }
\end{equation*}
Thus $\theta=\langle\langle \iota,\iota\rangle,\iota\rangle$. We claim that $\theta$ is not nullhomotopic.
Localise at $p=3$. According to \cite[Chapter 19]{Jam2}, there exists a $p$-local homotopy fibration
\begin{equation*}
S^3\overset{}{\hookrightarrow}G_2\longrightarrow S^{11}.
\end{equation*}
Thus by connectivity the map induced by the inclusion of the bottom cell into $G_2$ induces a homomorphism $i^{*}:\pi_m(S^3)\rightarrow \pi_m(G_2)$ which is an isomorphism for $m\leq 9$. By our previous argument, the map $\langle\langle\iota_3,\iota_3,\rangle,\iota_3\rangle\in\pi_9(S^3)$ is essential at $p=3$. Therefore the map $\theta$ has order 3. Thus $\theta$ generates $\pi_9(G_2)\cong\mathbb{Z}_3$. By definition, $\theta$ is the restriction of $\partial_1$ to $S^6\wedge S^3\subset S^6\wedge G_2$. Thus localising at 3, the order of $\partial_1$ is divisible by 3. Now, from Proposition \ref{bundles} we know that $Prin_{G_2}(S^7)=\pi_6(G_2)\cong\mathbb{Z}_3$. Thus, as there are 3 isomorphism classes of principal $G_2$-bundles over $S^7$, the order of the map $\partial_1$ is at most 3. The upper and the lower bounds of the order of $\partial_1$ coincide. Therefore the order of the connecting map $\partial^1$ is 3. Using an exact sequence as in \eqref{eseq} and the homotopy groups of spheres we obtain an exact sequence
$$\xymatrix{
\mathbb{Z}\ar[r]^{\partial_k}&\mathbb{Z}_3\ar[r]^-{i^{*}}&\pi_3({\rm{Map}}^k(S^7,BG_2))\ar[r]&0.
}$$
Therefore $| coker\partial_k|=(3,k)$. Thus if $\pi_3({\rm{Map}}^k(S^7,BG_2))\cong \pi_3({\rm{Map}}^{k'}(S^7,BG_2))$ then $(3,k)=(3,k').$

Finally a simple application of Lemma \ref{Th2L} shows that localised rationally or at any prime $\mathcal{G}_k(S^7)\simeq \mathcal{G}_{k'}(S^7)$ whenever $(3,k)=(3,k')$ and $G = G_2$. 

\item 
Suppose all spaces are localised at a prime $p\geq 3$. We can get an upper bound on the order of $\partial_1$ at a prime $p\geq 3$ as follows. Integrally, the attaching map of the 5-cell in $SU(3)$ has order 2 \cite{Mim}. Therefore, after localising at $p\geq3$, there exist $p$-local homotopy equivalences $$SU(3)\simeq S^3\times S^5$$
\begin{equation*}
\Sigma^6 SU(3)\simeq S^9\vee S^{11}\vee S^{14}
\end{equation*}
Using the previous equivalences we obtain
\begin{eqnarray*}
[\Sigma^6 SU(3), SU(3)] &=&[S^9\vee S^{11}\vee S^{14}, S^3\times S^5]\\
&=&\pi_9(S^3\times S^5)\oplus\pi_{11}(S^3\times S^5)\oplus\pi_{14}(S^3\times S^5)
\end{eqnarray*} 
From the homotopy groups of the spheres \cite{Tod} we obtain $[\Sigma^6SU(3),SU(3)]\cong\mathbb{Z}_3^2\oplus\mathbb{Z}_7.$ Let $\beta$ be the order of $\partial_1\in [\Sigma^6SU(3),SU(3)]$. Then $\beta$ divides $|\mathbb{Z}_3^2 \oplus \mathbb{Z}_7| = 63$. We also have that $\beta\leq |Prin_{SU(3)}(S^7)|=6.$ Therefore the order of $\partial_1$ localised at a prime $p\geq 3$ is at most 3.  

Localised at $p=3$, the composite $\iota: S^3 \hookrightarrow SU(3)$ is a generator of $\pi_3(SU(3))$. Let $\langle \iota,\iota\rangle$ be a generator of $\pi_6(S^3)\cong\pi_6(SU(3))\cong\mathbb{Z}_3$. Consider the composite
\begin{equation*} 
 \xymatrix@=0.8in{
 S^9\cong S^6\wedge S^3\ar[r]^{\mathbbm{1}\wedge \iota}&S^6\wedge SU(3)\ar[r]^-{\langle\langle\iota,\iota\rangle,{1}_{SU(3)}\rangle}&SU(3).
 }
\end{equation*}

The element $\langle\langle\iota,\iota\rangle ,\iota\rangle$ is non-trivial in $\pi_9(SU(3))\cong\mathbb{Z}_3$. Therefore localised at $p=3$ the map $\langle\langle\iota,\iota\rangle,\iota\rangle$ has order 3. Thus using an exact sequence as in \eqref{eseq}, we see that if $\mathcal{G}^k(S^7)\simeq\mathcal{G}^{k'}(S^7)$ then $(3,k)=(3,k')$. Finally, applying Lemma \ref{Th2L} we complete the proof of (2).
\item If $G\neq SU(2),SU(3), G_2$ or $Sp(1)$, then $\pi_7(BG)\cong 0$. Thus there is a single principal $G$-bundle over $S^7$ which must be the trivial bundle, implying that the map $\partial^1$ is nullhomotopic. Therefore the principal fibration 
\begin{equation*}
\Omega{\rm{Map}}_*(S^7,BG)\to \mathcal{G}^0(S^7)\to G
\end{equation*}
splits and $\mathcal{G}^0(S^7)\simeq \Omega^7G\times G$.\qedhere
\end{enumerate}
\end{proof}

\begin{remark}
In \cite{AK}, A. Kono obtained an integral classification of the homotopy types of $SU(2)$-gauge groups over $S^4$ by using information on the $p$-local homotopy types of the gauge groups for all primes $p$ along with fracture theorems for nilpotent spaces (see \cite[Chapter 13]{MP}). In the case of $SU(3)$ gauge groups over $S^7$,  arguing along the lines of \cite{AK} it would be also possible to upgrade Theorem \ref{t:S} for $G=SU(3)$ to an integral statement, if the order of the the connecting map $\partial^1$ at $p=2$ was known.
\end{remark}


\bibliographystyle{amsplain}
\bibliography{bibliography}
\end{document}